\documentclass[11pt]{amsart}

\usepackage{amsmath}
\usepackage{amsthm}
\usepackage{amscd}
\usepackage{amssymb}
\usepackage{graphicx}
\usepackage{latexsym} 
\usepackage{fullpage}
\usepackage{hyperref}
\usepackage{comment}

\newcommand{\cal}{\mathcal}

\def\epsilon{\varepsilon}
\def\phi{\varphi}

\newcommand{\R}{\mathbb R}
\newcommand{\Z}{\mathbb Z}

\newcommand{\N}{\mathbb N}

\def\strutdepth{\dp\strutbox}
\def \ss{\strut\vadjust{\kern-\strutdepth \sss}}
\def \sss{\vtop to \strutdepth{
\baselineskip\strutdepth\vss\llap{$\diamondsuit\;\;$}\null}}

\def\strutdepth{\dp\strutbox}
\def \sst{\strut\vadjust{\kern-\strutdepth \ssss}}
\def \ssss{\vtop to \strutdepth{
\baselineskip\strutdepth\vss\llap{$\spadesuit\;\;$}\null}}

\def\strutdepth{\dp\strutbox}
\def \ssh{\strut\vadjust{\kern-\strutdepth \sssh}}
\def \sssh{\vtop to \strutdepth{
\baselineskip\strutdepth\vss\llap{$\heartsuit\;\;$}\null}}

\def\qed{\hfill\rlap{$\sqcup$}$\sqcap$\par}

\def\strutdepth{\dp\strutbox}
\def \ss{\strut\vadjust{\kern-\strutdepth \sss}}
\def \sss{\vtop to \strutdepth{
\baselineskip\strutdepth\vss\llap{$\diamondsuit\;\;$}\null}}

\def\strutdepth{\dp\strutbox}
\def \sst{\strut\vadjust{\kern-\strutdepth \ssss}}
\def \ssss{\vtop to \strutdepth{
\baselineskip\strutdepth\vss\llap{$\spadesuit\;\;$}\null}}

\def\qed{\hfill\rlap{$\sqcup$}$\sqcap$\par}

\newtheorem{thm}{Theorem}[section]
\newtheorem{cor}[thm]{Corollary}
\newtheorem{lem}[thm]{Lemma}
\newtheorem{prop}[thm]{Proposition}

\theoremstyle{definition}
\newtheorem{defn}[thm]{Definition}

\newtheorem{convention-defn}[thm]{Convention-Definition}
\newtheorem{defn-rem}[thm]{Definition-Remark}

\newtheorem{rem}[thm]{Remark}

\newtheorem{quest}[thm]{Question}

\numberwithin{equation}{section}

\begin{document}

 \author[Martin Lustig]{Martin Lustig} 
 \address{\tt 
Aix Marseille Universit\'e, CNRS, Centrale Marseille, I2M UMR 7373,
13453  Marseille, France
}
 \email{\tt Martin.Lustig@univ-amu.fr} 
 
 \author[Caglar Uyanik]{Caglar Uyanik} 
 \address{\tt Department of Mathematics, University of Illinois at 
 Urbana-Champaign, 1409 West Green Street, Urbana, IL 61801, USA
 \newline \href{https://sites.google.com/site/caglaruyanik/}{https://sites.google.com/site/caglaruyanik/}} 
 \email{\tt cuyanik2@illinois.edu}

\title[Perron-Frobenius for reducible substitutions]{Perron-Frobenius theory and frequency convergence for reducible substitutions}

\begin{abstract}
We prove a general version of the classical Perron-Frobenius convergence property for reducible matrices. We then 
apply this result to reducible substitutions and use it to produce limit frequencies for factors and hence invariant measures on the associated subshift. 

The analogous results 
are well known for primitive substitutions and 
have found many applications, but for reducible substitutions the tools provided here were so far missing from the theory.
\end{abstract} 

 


\maketitle
 

\section{Introduction}
\label{introduction}

One of the most investigated dynamical systems, with important applications in many areas, are subshifts that are generated by substitutions.  If the substitution is primitive, then a number of well known and powerful tools are available, most notably the Perron-Frobenius theorem for primitive matrices, which ensures that the subshift in question is uniquely ergodic.

On the other hand, substitutions with reducible incidence matrices 
have only recently received some serious attention (see Remark \ref{moulinette}
and Remark \ref{Schneider}).
One reason for this neglect is that the standard methods, employed in the primitive case 
for 
analyzing the dynamics of
such substitutions and their incidence matrices,
uses tools that so far didn't have analogues in reducible case.  
It is the purpose of this paper to provide these tools, and thus 
to extend the basic theory from the primitive to the reducible case.

We concentrate on substitutions $\zeta$ which are {\em expanding}, i.e. 
$\zeta$ 
does not act periodically or erasing
on any subset of the given alphabet
(for our notation and terminology on substitutions see \S \ref{substitutions}).

\smallskip

Every non-negative irreducible square matrix has a power which is a block diagonal matrix, where every diagonal block is primitive.
The classical Perron-Frobenius theorem asserts that, for any primitive matrix $M$ and for any non-negative column vector $\vec{v}\neq\vec{0}$, 
the sequence of vectors  $M^t \vec v$,
after normalization, converges to a
positive eigenvector of $M$, and that the latter is unique up to rescaling.

In analogy 
with  
the above
facts, 
in section \ref{prelims*} we 
introduce
the  
{\em PB-Frobenius form} for matrices, 
which is set up so that, 
up to conjugation with a permutation matrix, every non-negative integer square 
matrix has  
a positive power which is in PB-Frobenius form. We 
prove the following 
convergence
result
for matrices in PB-Frobenius form;
its proof 
spans 
sections \ref{primitive-Frobenius}--\ref{PB-convergence}
and 
can be read independently 
from 
the rest of the paper.

\begin{thm}
\label{thmI}
Let $M$ be a non-negative integer $(n\times n)$-matrix which is in PB-Frobenius form.
Assume that none of the coordinate vectors is mapped by a positive power of $M$ to itself or to $\vec 0$.

Then for any non-negative 
column 
vector $\vec{v} \neq \vec 0$ 
there exists a ``limit vector'' 
\[
\vec{v}_{\infty}=\lim_{t \to \infty} \frac{1}{\|M^t \vec v\|} M^t \vec v
\neq \vec 0\, ,
\]
and $\vec{v}_{\infty}$ 
is an 
eigenvector of $M$. 
\end{thm}

In symbolic dynamics 
the classical Perron-Frobenius theorem plays a 
key role, when applied to the incidence matrix $M_\zeta$ of a primitive substitution $\zeta$:
Any finite word $w$ in the language 
$\cal L_\zeta$ 
associated to 
$\zeta: A \to A^*$ has the property that for any letter $a_i$ of the alphabet $A$, the number 
$|\zeta^t(a_i)|_w$ 
of occurrences of $w$ 
as a factor 
in $\zeta^t(a_i)$, normalized by the word length $|\zeta^t(a_i)|$, converges to a well defined {\em limit frequency}. The latter can 
be used to 
define the unique 
(up to scaling) invariant measure on the subshift 
$\Sigma_\zeta$ 
defined by 
the primitive substitution 
$\zeta$.

\medskip

The purpose of this paper is to 
establish 
the analogous results for expanding reducible substitutions $\zeta$. The key observation (Proposition \ref{blow-up-Frobenius}) here is that for any $n \geq 2$  the classical 
{\em level $n$ blow-up substitution} $\zeta_n$ (based on a derived alphabet $A_n$ which contains all factors 
$w_i \in \cal L_\zeta$
of length 
$|w_i| = n$ 
as ``blow-up letters'') has 
incidence 
matrix 
$M_{\zeta_n}$ 
in PB-Frobenius form, assuming that 
the 
incidence matrix 
$M_\zeta$ is in PB-Frobenius form. 

Combining Proposition \ref{blow-up-Frobenius} with Theorem \ref{thmI} 
gives the following
(see 
Lemma \ref{letter-convergence} and 
Proposition 
\ref{theoremB}):

\begin{thm}
\label{theoremII}
Let $\xi$ be an expanding substitution 
on 
a finite alphabet $A$. 
Then there exist a positive power 
$\zeta=\xi^{s}$ 
such that
for 
any 
non-empty 
word 
$w \in A^*$ 
and any letter $a_i \in A$ 
the 
limit 
frequency
\[
\lim_{t\to\infty}\frac{|\zeta^{t}(a_i)|_w}{|\zeta^{t}(a_i)|}
\]
exists. 
\end{thm}

As a consequence 
of Theorem \ref{theoremII} we 
obtain 
 - precisely as in the primitive case -
for any $a_i \in A$
an invariant measure 
on the subshift $\Sigma_\zeta$ defined by the substitution $\zeta$.
However, contrary to the primitive case, in general 
this invariant measure 
will heavily depend on the chosen letter $a_i$, 
see Question \ref{yet-open}. 
We prove (see 
Remark \ref{invariant-measure+}):

\begin{cor}
\label{invariant-measure1}
For any expanding substitution $\zeta: A \to A^*$ and any letter 
$a_i \in A$ there is a well defined invariant measure $\mu_{a_i}$ on the substitution subshift $\Sigma_\zeta$. 
For any 
non-empty $w \in A^*$ and the associated 
cylinder ${\rm Cyl}_w \subset \Sigma_\zeta$ 
(see subsection \ref{invmeasures})
the value of $\mu_{a_i}$
is
given, 
after possibly raising $\zeta$ to a suitable power according to Theorem \ref{theoremII},
by the limit frequency 
\[
\mu_{a_i}({\rm Cyl}_w) = \lim_{t\to\infty}\frac{|\zeta^{t}(a_i)|_w}{|\zeta^{t}(a_i)|}.
\]
\end{cor}

Although there are various generalizations of the classical Perron-Frobenius theorem for primitive matrices in the literature, we could not find one with the convergence statement as in Theorem \ref{thmI}, which is needed for our applications. 
Perron-Frobenius theory and its generalizations are relevant in many more branches of mathematics than just symbolic dynamics, including applied linear algebra, and some areas of analysis and probability theory (see for instance \cite{AGN},\cite{BSS} and \cite{Le}). 
We expect that Theorem \ref{thmI} will find useful applications in other contexts. 

Our proof of Theorem \ref{thmI} uses only standard methods from linear algebra and is hence accessible to mathematicians from all branches. The reader interested only in Theorem \ref{thmI} may go straight to section \ref{primitive-Frobenius} and start reading from there. 
The sections \ref{primitive-Frobenius} to \ref{PB-convergence} 
are organized as follows:

After setting up some definitions and terminology in section \ref{primitive-Frobenius}, we state Theorem \ref{GPFT}, a 
slight 
strengthening of Theorem \ref{thmI}. 
To stay within the realm if this paper 
we phrase 
Theorem \ref{GPFT} 
for integer matrices, but this assumption is 
not 
used in the proof of Theorem \ref{GPFT}.

The proof of Theorem \ref{GPFT} is done by induction over the number of  primitive diagonal blocks in a suitable power of the given matrix $M$, and the induction step itself (Proposition \ref{induction-step}) reveals a crucial amount of information about the 
dynamics on the non-negative cone $\R_{\geq 0}^n$ induced by iterating the map which is defined by the matrix $M$.  
The proof of 
Proposition \ref{induction-step}, which involves a careful (and hence a bit lengthy) 3-case analysis, is assembled in section \ref{theproof}. 
In 
section \ref{primitive-F-eigenvectors} some results about the eigenvectors of such a matrix $M$ are 
shown to be 
direct consequence of Proposition~\ref{induction-step}.

\medskip
\noindent
{\em Acknowledgements:} We would like to thank Ilya Kapovich and Chris Leininger for their interest and helpful discussions related to this work. We also want to thank Arnaud Hilion and Nicolas B\'edaride for several useful comments and remarks. Finally, we would like thank Jon Chaika for helpful conversations and useful suggestions. 

Both authors gratefully acknowledge support from U.S. National Science Foundation grants DMS 1107452, 1107263, 1107367 "RNMS: GEometric structures And Representation varieties" (the GEAR Network)." 
The first author was partially supported by the French research grant ANR-2010-BLAN-116-01 GGAA. 
The second author was partially supported by the NSF grants of Ilya Kapovich (DMS 1405146) and Christopher J. Leininger (DMS 1510034).

\section{Non-negative matrices in PB-Frobenius form}
\label{prelims*}

A non-negative
integer
$(n\times n)$-matrix $M$ is called \emph{irreducible} if for any $1\le i,j\le n$ there exists an exponent $k=k(i,j)$ such that the $(i, j)$-th 
entry of $M^k$ 
is positive.
The matrix $M$ is called \emph{primitive} if the 
exponent $k$ can be chosen independent of $i$ and $j$. 
The matrix $M$ is called \emph{reducible} if $M$ is not irreducible. 
Since in some places in the literature the $(1 \times 1)$-matrix with entry 0 is also accepted as ``primitive'' 
we will be explicit whenever this issue comes up.

It is a well known fact for non-negative matrices that every irreducible matrix has a power which is, 
up to conjugation with a permutation matrix, a block diagonal matrix where every diagonal block is 
a primitive square matrix.

For the purposes of 
our results on reducible substitutions presented in the next section 
the following terminology turns out to be 
crucial:

\begin{defn}
\label{power-bounded}
A non-negative integer square matrix $M$ is called \emph{power bounded ($PB$)} if the entries of $M^{t}$ are uniformly bounded for all $t\ge1$.
\end{defn}

Let $M$ be a non-negative integer square matrix as considered above, and assume  
that $M$ is partitioned into matrix blocks which along the diagonal are square matrices.

\begin{defn}
\label{Frobenius-form*}
(a)
The matrix $M$ is in {\em PB-Frobenius form} if $M$ is a lower diagonal block matrix where every diagonal block is either primitive or power bounded.

\smallskip
\noindent
(b)
If $M$ is in PB-Frobenius form, then 
the special case of a diagonal block which is a $(1 \times 1)$-matrix with entry $1$ or $0$ will be counted as PB block and not as primitive block, although technically speaking such a  
block
could also be considered as ``primitive''.\end{defn}

\begin{lem}
\label{Frobenius-powers*}
Every non-negative square matrix $M$ has a positive power $M^t$ which
is in PB-Frobenius form
(with respect to some block decomposition of $M$).
\end{lem}

\begin{proof}
This is an immediate consequence of the well known normal form for non-negative matrices, which says that, up to conjugation with a permutation matrix, $M$ is a lower block diagonal matrix with all diagonal blocks are either zero or irreducible. It suffices now to rise $M$ to a power such that every diagonal block matrix block is itself a block diagonal matrix with primitive matrix blocks, and to refine the block structure of $M$ accordingly.
\end{proof}

As is often done when working with non-negative matrices, we will use 
in this paper 
as norm on $\R^n$ the $\ell_1$-norm, i.e. 
$$\big{\|} \sum a_i \vec e_i  \big{\|} = \sum |a_i|$$
for all $a_1 , \ldots , a_n \in \R$.

In section \ref{PB-convergence}
we prove the 
convergence result
for matrices in PB-Frobenius form 
stated in Theorem~\ref{thmI}, 
which is crucial for our extension of the classical theory for primitive substitutions to 
the much more general class of 
expanding 
substitutions in the next 
section. 
It turns out 
(see Proposition \ref{blow-up-Frobenius})
that the class of PB-Frobenius matrices is precisely the class of matrices for which the blow-up technique known from primitive matrices can be extended naturally. 



For practical purposes we formalize the condition that is used as assumption in 
Theorem \ref{thmI}:

\begin{defn-rem}
\label{expanding-matrix}
(1) An integer square matrix 
$M$ is called {\em expanding} if 
none of the coordinate vectors $\vec e_i$
is mapped by a positive power of $M$ to itself 
or to $\vec 0$.

\smallskip
\noindent
(2)
It is easy to see that this is equivalent to the condition that for any non-negative column vector $\vec v \neq \vec 0$ the length of the iterates 
satisfy
\[
\|M^t \vec v\| \to \infty
\]
for $t \to \infty$.

\smallskip
\noindent
(3)
Let $M$ be in PB-Frobenius form.  The statement that ``$M$ is expanding'' is equivalent to the requirement that no minimal 
diagonal matrix block $M_{i, i}$ 
of $M$ is PB. Here {\em minimal} refers to the partial order on blocks 
as defined in section \ref{primitive-Frobenius}.
Thus ``$M_{i, i}$ is minimal'' means that
$M \vec v$ has non-zero coefficients only in 
the coordinates corresponding to $M_{i, i}$,
if the same assertion is true for $\vec v$.
\end{defn-rem}

\section{Dynamics of expanding substitutions}
\label{dynamics-section}
\subsection{Basics of substitutions}
\label{substitutions}
A \emph{substitution} $\zeta$ on a finite set $A=\{a_1,a_2,\ldots a_n\}$ (called the {\em alphabet}) of {\em letters} $a_i$ is given by associating to every $a_i \in A$ a finite 
word $\zeta(a_i)$ in the {alphabet} $A$:
$$
a_i \mapsto \zeta(a_i) = x_1 \ldots x_n \quad \quad {\rm (with} \quad x_i \in A)
$$
This defines
a map from $A$ to $A^{*}$, by which we denote the free monoid over the alphabet $A$. The map $\zeta$ extends to a well defined monoid endomorphism $\zeta: A^* \to A^*$ which is usually denoted by the same symbol as the substitution.

The combinatorial length of $\zeta(a_i)$, denoted by $|\zeta(a_i)|$, is the number of letters in the word $\zeta(a_i)$. 
We call a substitution $\zeta$ 
\emph{expanding} if there exists $k\ge1$ such that for every $a_i\in A$
one has
\[
|\zeta^{k}(a_i)|\ge2.
\]
It follows directly that this is equivalent to stating that $\zeta$ is {\em non-erasing}, i.e. none of the $\zeta(a_i)$ is equal to the empty word, and that $\zeta$ doesn't act periodically on any subset of the generators.

Let $A^{\mathbb{Z}}$ be the set of all biinfinite words $\ldots x_{-1}x_{0}x_{1}x_2\ldots $ in $A$, 
endowed with the 
product
topology.  It is equipped with the {\em shift operator}, which shifts the indices of any biinfinite word by $-1$, and is continuous. 

Any substitution $\zeta$ defines a {\em language} $\cal L_\zeta \subset A^*$ which consists of all words $w \in A^*$ that appear as a factor of $\zeta^{k}(a_i)$ for some $a_i \in A$ and some $k\ge 0$. 
Here {\em factor} means any finite subword of a word in $A^*$ or $A^\Z$, referring to the multiplication in the free monoid $A^*$.

Furthermore, $\zeta$ defines a {\em substitution subshift}, i.e. a  subshift $\Sigma_\zeta \subset A^\Z$ which is 
the space of all biinfinite words in $A$ which have the property that any finite factor belongs to $\cal L_\zeta$.

\smallskip

A substitution $\zeta$ on $A$ is called \emph{irreducible} if for all $1\le i,j\le n$, there exist $k=k(i,j)\ge1$ such that $\zeta^{k}(a_j)$ contains the letter $a_i$. It is called {\em primitive} if $k$ can be chosen independent of $i,j$. A substitution is called \emph{reducible} if it is not irreducible. 
Note that any irreducible substitution $\zeta$ (and hence any primitive $\zeta$) is expanding, except if $A = \{a_1\}$ and $\zeta(a_1) = a_1$.

Given a substitution $\zeta:A\to A^*$, there is an associated \emph{incidence matrix} $M_\zeta$ defined as follows: The $(i,j)^{th}$ entry of $M_{\zeta}$ is the number of occurrences of the letter $a_i$ in the word $\zeta(a_j)$. Note that the matrix $M_\zeta$ is a non-negative integer square matrix. It is easy to verify that 
an expanding 
substitution $\zeta$ is irreducible (primitive) if and only if the matrix $M_\zeta$ is irreducible (primitive), as defined in section \ref{prelims*}.

It also follows directly that $M_{\zeta^t} = (M_\zeta)^t$ for any exponent $t \in \N$. Furthermore, the incidence matrix $M_\zeta$ is expanding 
(see Definition-Remark \ref{expanding-matrix})
if and only if the substitution $\zeta$ is expanding. 
In particular, we obtain directly from Lemma \ref{Frobenius-powers*}:

\begin{lem}
\label{PB-power-substitution}
Every expanding substitution $\zeta$ has a positive power $\zeta^t$ such that the incidence matrix $M_{\zeta^t}$ is PB-Frobenius
and expanding.
\qed
\end{lem}

\subsection{Frequencies of letters} 

For any letter $a_i\in A$ and any word $w\in A^*$ we denote the number of occurrences of the letter $a_i$ in the word $w$ by 
$|w|_{a_i}$.

We observe directly from the definitions that 
the resulting 
{\em occurrence vector}
$\vec v(w) := (|w|_{a_i})_{a_i \in A}$ satisfies:
\begin{equation}
\label{3.1}
M_\zeta \cdot \vec v(w) = \vec v({\zeta(w)})
\end{equation}

The statement of the following lemma, for the special case of primitive substitutions, is a well known classical tool in symbolic dynamics (see \cite[Proposition 5.8]{Q}).

\begin{lem} 
\label{letter-convergence}
Let $\zeta: A^*\to A^*$ be an expanding substitution. Then, up to replacing $\zeta$ by a positive power, for any $a \in A$ and any $a_i\in A$ the 
limit frequency
\[
f_{a_i}(a):= \lim_{t\to\infty}\frac{|\zeta^{t}(a)|_{a_i}}{|\zeta^{t}(a)|}
\]
exists. 
The resulting limit 
frequency 
vector $\vec v_{\infty}(a) := (f_{a_i}(a))_{a_i \in A}$ is an eigenvector of the matrix $M_\zeta$.
\end{lem} 

\begin{proof} 
By Lemma \ref{PB-power-substitution} we can assume that,
up to replacing $\zeta$ by a positive power, the incidence matrix $M_\zeta$ is in PB-Frobenius form
and expanding.
Thus, Theorem \ref{thmI} applied to the 
occurrence vector $\vec v({a})$ 
gives the required result, where we note that $\| M_\zeta^t\vec v({a}) \| = \|\vec v(\zeta^t(a))\| = |\zeta^{t}(a)|$ is a direct consequence of 
equality (\ref{3.1}) and 
the definition of the norm in section \ref{prelims*}.
\end{proof}

Notice that, as for primitive substitutions, 
it follows that the sum of the coefficients of the limit 
frequency 
vector $\vec v_{\infty}(a)$ is equal to $1$.  However, contrary to the primitive case, for a reducible substitution $\zeta$ the limit 
frequency 
vector $\vec v_\infty(a)$ will in general depend on the choice of $a \in A$.

\begin{rem}
\label{new-remark}
From the statement of Lemma \ref{letter-convergence} and from equality (\ref{3.1}) one obtains directly that $f'_{a_i}(a):= \underset{t\to\infty}{\lim}\frac{|\zeta^{t+1}(a)|_{a_i}}{|\zeta^{t}(a)|}$ exists and that it gives rise to a vector $\vec v'_{\infty}(a) := (f'_{a_i}(a))_{a_i \in A}$ which satisfies $\vec v'_{\infty}(a) = M_\zeta \vec v_{\infty}(a)$. Since $\vec v_{\infty}(a)$ is an eigenvector of $M_\zeta$, with eigenvalue that satisfies $\lambda_a > 1$, we deduce from $f_{a_i}(a):= \underset{t\to\infty}{\lim}\frac{|\zeta^{t}(a)|_{a_i}}{|\zeta^{t}(a)|} = \underset{t\to\infty}{\lim}\frac{|\zeta^{t+1}(a)|_{a_i}}{|\zeta^{t+1}(a)|}$ that
$$
\lim_{t\to\infty}\frac{|\zeta^{t+1}(a)|}{|\zeta^{t}(a)|} = \lambda_a \,.
$$
\end{rem}

\subsection{Frequencies of factors via the level $n$ blow-up substitution} 

Recall from section \ref{substitutions} that for any substitution $\zeta$ we denote by $\cal L_\zeta$ the subset of $A^*$ which consists of all factors of any iterate $\zeta^{k}(a_i)$, for any letter $a_i \in A$.
We say that $w$ is 
{\em used} by $a_i$ if $w$ appears as a factor in some $\zeta^{k}(a_i)$.

We see from Lemma \ref{letter-convergence} that the frequencies of letters are encoded in the incidence matrix $M_\zeta$; 
however, 
this matrix doesn't give us 
any information about the frequencies of factors. 
In order to understand the asymptotic behavior of frequencies of factors one has to appeal to a classical ``blow-up'' technique for the substitution (see for instance \cite{Q}). We now give a quick introduction to this blow-up technique, which will be crucially used below.
 
Let $n\ge2$, 
and denote by 
$A_n = A_n(\zeta)$ the set of all 
words in $\cal L_\zeta$ 
of length $n$.
We consider $A_n$ as the new alphabet, and define a substitution $\zeta_n$ on $A_n$ as follows: 

For $w=a_1 a_2\ldots a_{n}\in A_n$, consider the word 
\[
\zeta(a_1a_2\ldots a_{n})=x_1x_2\ldots x_{|\zeta(a_1)|}x_{|\zeta(a_1)|+1}\ldots x_{|\zeta(w)|}. 
\]

Define 
\[
\zeta_n(w)=(x_1\ldots x_{n})(x_2\ldots x_{n+1})\ldots(x_{|\zeta(a_1)|}\ldots x_{|\zeta(a_1)|+ n-1}).
\]

That is, $\zeta_{n}(w)$ is defined as the ordered list of first $|\zeta(a_1)|$ 
factors of length $n$ of the word $\zeta(w)$. As before, $\zeta_n$ extends to $A_n^*$ and 
$A_n^{\mathbb{Z}}$, by concatenation. 
Here a word $w'\in A_n^*$ of length $k$ is an ordered list of $k$ words of length $n$ in $A^*$. Namely,
\[
w'=w_0w_1\ldots w_k
\]
such that $|w_i|=n$ for all $i=1,\ldots, k$. 
We call $\zeta_n$ the \emph{level $n$ blow-up substitution} for $\zeta$. From this definition
it follows directly 
that $(\zeta_n)^{t}=(\zeta^{t})_n$, hence we will omit the parentheses. 
Observe that for $w=a_1 a_2\ldots a_{n}\in A_{n}$, we have $|\zeta_{n}(w)|=|\zeta(a_1)|$, from which 
it follows that for an expanding substitution $\zeta$ the blow-up substitution $\zeta_n$ is expanding, for any $n \geq 2$.

One of the classical tools that is 
used to understand irreducible 
substitutions and their invariant measures is the following:

\begin{lem}\cite[Lemma 5.3]{Q}
\label{blow-up-primitive}
Let $\zeta:A^*\to A^*$ be a substitution such that $M_\zeta$ is primitive. Then for any $n\ge 1$, the incidence matrix $M_{\zeta_n}$ for the level $n$ blow-up substitution $\zeta_n$ is again primitive. 
\end{lem} 

We show that the analogue is true for 
expanding 
substitutions with 
possibly
reducible incidence matrices:

\begin{prop} 
\label{blow-up-Frobenius}
Let $\zeta:A^*\to A^*$ be a substitution such that $M_\zeta$ is in PB-Frobenius form. Then for any $n\ge1$, the incidence matrix $M_{\zeta_n}$ for the level $n$ blow-up substitution $\zeta_n$ is again in PB-Frobenius form. 
\end{prop} 

The proof of this proposition, 
which is one of the main results of this paper, 
requires several lemmas; we assemble all of 
 them in the next subsection.

\subsection{The proof of Proposition \ref{blow-up-Frobenius}}
\label{blow-up}

Let $\zeta: A^* \to A^*$ be a substitution as before, and let $A' \subset A$ be a {\em $\zeta$-invariant subalphabet}, i.e. we assume that $\zeta(a') \in A'^*$ for any $a' \in A'$, where we identify the free monoid $A'^*$ with the submonoid of $A^*$ that is generated by the letters from $A'$. 

For most applications one may chose $A'$ to be a maximal proper $\zeta$-invariant subalphabet of $A$, although formally we don't need this assumption.
The terminology below comes from thinking of $A \smallsetminus A'$ as representing the ``top stratum'' for the reducible substitution $\zeta$.

For any $n \geq 2$ and for the level $n$ blow-up substitution $\zeta_n: A_n \to A_n^*$ we consider the subalphabet $A'_n \subset A_n$ which is given by all words $w = x_1 \ldots x_n$ with $x_i \in A'$
that are used by some $a'_i \in A'$.
From the $\zeta$-invariance of $A'$ it follows directly that $A'_n$ is $\zeta_n$-invariant.

We now partition the letters $w$ of $A_n \smallsetminus A'_n$, i.e. 
$w = x_1 \ldots x_n$ is a word of length $n$ 
which 
is 
used by some $a_i \in A \smallsetminus A'$ but not by any $a'_i \in A'$,
into two classes:
\begin{enumerate}
\item
$w \in A_n \smallsetminus A'_n$ is {\em top-used} if $x_1 \in A \smallsetminus A'$.
\item
$w \in A_n \smallsetminus A'_n$ is {\em top-transition} if 
$x_1 \in A'$.
\end{enumerate}

\begin{rem}
\label{transition-invariance}
From the definition of the map $\zeta_n$ and from the $\zeta$-invariance of $A'$ it follows directly that the top-transition words together with $A'_n$ constitute a $\zeta_n$-invariant subalphabet of $A_n$. Indeed,
recall that for any $w = x_1 \ldots x_n \in A_n$ the image $\zeta^t_n(w)$ is a word $w_1 w_2 \ldots w_r$ in $A_n$, with $r = r(t) = |\zeta^t(x_1)|$ such that $w_k$ is the prefix of length $n$ of the word obtained from $\zeta^t(w)$ by deleting the first $k-1$ letters.
Thus it follows that the first $r - (n-1)$ of the words $w_k$
are factors of $\zeta^t(x_1)$, and that the last $n-1$ of the words $w_k$ have at least their first letter in $\zeta^t(x_1)$.
Hence, 
if $x_1 \in A'$, then 
the first $r - (n-1)$ of the words $w_k$ 
belong to $A'_n$, and the last $n-1$ words $w_k$ are all top-transition.
\end{rem} 

We now consider the incidence matrices $M_\zeta$ and $M_{\zeta_n}$:  From the $\zeta$-invariance of $A'$ it follows that after properly reordering the letters of $A$ the matrix $M_\zeta$ is a $2\times 2$ lower triangular block matrix, with $M_{\zeta|_{A'}}$ as lower diagonal block.  Similarly, $M_{\zeta_n}$ is a $3 \times 3$ lower triangular block matrix, with $M_{\zeta_n|_{A'_n}}$ as bottom diagonal block. 
The top-used edges form the top diagonal block, and top-transition edges form the 
middle 
diagonal block.

The arguments given below work also for the special case where $A'$ is empty; in this case the bottom diagonal block of $M_\zeta$ and the two bottom diagonal blocks of $M_{\zeta_n}$ 
have size $0 \times 0$, so that both, $M_\zeta$ and $M_{\zeta_n}$, consist de facto of a single matrix block.

\begin{lem}
\label{third-block}
The middle diagonal block of $M_{\zeta_n}$ as defined above is power bounded.
\end{lem}

\begin{proof}
Using the same terminology as in Remark \ref{transition-invariance} we recall that for $w = x_1 \ldots x_n \in A_n$ and $\zeta^t_n(w) = w_1 w_2 \ldots w_{|\zeta^t(x_1)|}$ it follows from $x_1 \in A'$ that
only the last $n-1$ words $w_k$ may possibly be in $A_n \smallsetminus A'_n$, but their first letter always belongs to $A'$. This shows that independently of $t$ any coefficient in the middle diagonal block of $M_{\zeta^t_n}$ is bounded above by $n$, for any $t \geq 1$.
\end{proof}

\begin{lem}
\label{non-growing-top}
If the top block diagonal matrix of $M_\zeta$ 
is power bounded, then so is the top 
block diagonal matrix of $M_{\zeta_n}$.
\end{lem}

\begin{proof}
From the hypothesis that the top block diagonal matrix of $M_{\zeta}$ is power bounded we obtain that there is a constant $K \in \N$ such that for any letter $a_i \in A \smallsetminus A'$ and any $t \geq 0$ the number of letters $x_i$ of the factor $\zeta^t(a_i)$ that do not belong to $A'$ is bounded above by $K$. But then it follows directly that there can be at most $K$ top-used letters $y_1 \ldots y_n$  from $A_n \smallsetminus A'_n$ in any of the $\zeta_n^t(w)$ with $w = x_1 \ldots x_n \in A_n \smallsetminus A'_n$ top-used, since any such 
$y_1 \ldots y_n$ must have its initial letter $y_1$ in $\zeta^t(x_1)$, and $y_1$ must belong to $A \smallsetminus A'$.
\end{proof}

\begin{rem}
\label{observation}
From the definition of ``top-used'' and from the finiteness 
of $A_n$ it follows that there is an exponent $t \geq 0$ such that for any 
word $u \in A_n \smallsetminus A'_n$ 
(and hence in particular for any top-used $u$) 
there is a letter $a_i \in A \smallsetminus A'$ such that $u$ is a factor of the word $\zeta^{t'}(a_i)$ for some positive integer $t' \leq t$.
\end{rem}

\begin{lem}
\label{middle-block}
If the top diagonal block matrix of $M_\zeta$ is primitive, then so is the top diagonal block matrix of $M_{\zeta_n}$.
\end{lem}

\begin{proof}
It suffices to show that there is an integer $t_0 \geq 0$ such that for any two top-used words $w = x_1 \ldots x_n$ and $w'$ of $A_n$ the word $w'$ is a factor of the prefix 
of length $|\zeta^{t_0}(x_1)|$ of $\zeta^{t_0}(w)$.
From the assumption that the top diagonal block matrix of $M_\zeta$ is primitive we know that there is an exponent $t_1 \geq 0$ such that for any two letters $a$ and $a'$ of $A \smallsetminus A'$ the word $\zeta^{t'_1}(a')$ 
contains as factor the letter $a$, 
for any integer $t'_1 \geq t_1$. From the observation stated 
in Remark \ref{observation} we deduce
that there is an exponent $t_2 \geq 0$ such that $w'$ is a factor of $\zeta^{t'_2}(a'')$ of some letter $a''$ of $A \smallsetminus A'$, for some positive integer $t'_2 \leq t_2$. 
Thus from setting $a = x_1$ and $a' = a''$
it follows that $w'$ is a factor of $\zeta^{t_1 + t_2}(x_1)$. This shows the claim, for $t_0 = t_1 + t_2$.
\end{proof}
We now obtain as direct consequence of the above Lemmas:

\begin{proof}[Proof of Proposition \ref{blow-up-Frobenius}]
The claim that 
the incidence matrix $M_{\zeta_n}$ is in PB-Frobenius form
follows from an easy inductive argument over the number of blocks in the PB-Frobenius form of  $M_\zeta$: At each induction step the top left diagonal block of $M_\zeta$ is either primitive or power bounded, and all other blocks are assembled together in an invariant subalphabet $A'$ of the given alphabet $A$. 
Then $M_{\zeta_n}$ is considered as above as $3 \times 3$ lower triangular block matrix. For the two upper diagonal blocks the claim follows directly from the above lemmas. The bottom diagonal block is equal to $\zeta_n|_{A'_n}$, which is equal to the incidence matrix of $(\zeta|_{A'})_n$. But for $\zeta|_{A'}$ the claim can be assumed to be true via the induction hypothesis. 
\end{proof}

\subsection{Level $n$ limit frequencies}
\label{level-n-freq}
We can now state the analogue of Lemma \ref{letter-convergence} for words $w$ of length $n\geq 2$ instead of letters $a_i \in A$. As done there for $n=1$, we can use all words $w$ from the alphabet 
$A_n = A_n(\zeta)$ as ``coordinates'' and consider, for any 
word $w' \in A_{n}^*$, the 
{\em level $n$  
occurrence  
vector}
$\vec v_n(w') := (|w'|_w)_{w \in A_n}$. Again we obtain:
$$
M_{\zeta_n} \cdot \vec v_n(w') = \vec v_n({\zeta_n(w'))}
$$

\begin{prop}
\label{theoremB} 
Let $\zeta:A\to A^{*}$ be an expanding substitution.
Then, up to replacing $\zeta$ by a power, the frequencies of factors converge: For any 
word 
$w\in A^{*}$ 
of length 
$|w|\geq 2$ and any letter $a \in A$ the limit 
frequency 
\[
f_w(a) := \lim_{t\to\infty}\frac{|\zeta^{t}(a)|_w}{|\zeta^{t}(a)|}
\]
exists. 
\end{prop}
 
\begin{proof} 
Set $n = |w|$.
If $w$ does not belong to $A_n$, then $|\zeta^{t}(a)|_w = 0$ for all $t \in \N$, so that we can assume $w \in A_n$.

By Lemma \ref{PB-power-substitution} we can assume that,
up to replacing $\zeta$ by a positive power, the incidence matrix $M_\zeta$ is in PB-Frobenius form. Thus we can apply Proposition \ref{blow-up-Frobenius} to obtain that the blow-up incidence matrix $M_{\zeta_n}$ is also in PB-Frobenius form. Furthermore, if $\zeta$ is expanding,  
then so is $\zeta_n$, and hence $M_{\zeta_n}$.

From the definition of $\zeta_{n}$ we have the following estimate: For any two 
$w, w_1 \in A_n$, with 
$w_1 = x_1 x_2 \ldots x_n$, we have 
\[
\left||\zeta_{n}^{t}(w_1)|_w-|\zeta^{t}(x_1)|_w\right|\le n
\]
for all $t\ge1$. On the other hand, 
we have 
$|\zeta_{n}^{t}(w_1)|=|\zeta^{t}(x_1)|$. Therefore
one deduces:
\[
\lim_{t\to\infty}\left|\frac{|\zeta_{n}^{t}(w_1)|_w}{|\zeta_{n}^{t}(w_1)|}-\frac{|\zeta^{t}(x_1)|_w}{|\zeta^{t}(x_1)|}\right|\le \lim_{t\to\infty}\left|\frac{n}{|\zeta^{t}(x_1)|}\right|= 0
\]

Now, let $w'\in A_n$ be a word of length $n$ that starts with the letter $a$. As in the proof of Lemma \ref{letter-convergence} we can thus use Theorem \ref{thmI}, which applied to the level $n$ 
occurrence  
vector $\vec v_n(w')$ gives that
\[
\lim_{t\to\infty}\frac{|\zeta_{n}^{t}(w')|_w}{|\zeta_{n}^{t}(w')|}
\]
exists, and together with the above observation equals to 
\[
\lim_{t\to\infty}\frac{|\zeta^{t}(a)|_w}{|\zeta^{t}(a)|}.
\]
 
\end{proof}

Similar to the case where $n =1$ in Lemma \ref{letter-convergence} 
it follows that the sum of the coefficients of the limit 
frequency 
vector $\vec v_n^{\, \infty}(a)$ is equal to $1$.  Again, for an expanding reducible substitution $\zeta$ the limit 
frequency 
vector $\vec v_n^{\, \infty}(a)$ will in general depend on the choice of $a \in A$.

\subsection{Invariant measures for expanding substitutions}
\label{invmeasures}

Recall from section \ref{substitutions} that
the subshift $\Sigma_\zeta$ associated to a substitution $\zeta$ is
the space of all biinfinite words 
which have the property that any finite factor belongs to $\cal L_\zeta$.
Any word $w = x_1 \ldots x_m \in A^*$ defines a {\em Cylinder} 
${\rm Cyl}_w = {\rm Cyl}_w(\zeta) \subset \Sigma_\zeta$ which consists of all biinfinite sequences $\ldots y_i y_{i+1} y_{i+2} \ldots$ 
in $\Sigma_\zeta$ 
which satisfy $y_1 = x_1, y_2 = x_2, \ldots, y_m = x_m$.

In the classical case where $\zeta$ is 
primitive,
it is well known that the subshift $\Sigma_\zeta$ defined by $\zeta$ is uniquely ergodic. In this case
the limit 
frequency 
$f_w(a)$ obtained in Proposition \ref{theoremB} 
is typically 
used to describe the value that the invariant 
probability 
measure $\mu_\zeta$ takes on the cylinder ${\rm Cyl}_w$ defined by 
any 
$w \in \cal L_\zeta \subset A^*$ 
(see section 5.4.2 of \cite{Q}).

In 
the situation treated 
in this paper,
where $\zeta$ is only assumed to be expanding (so that $M_\zeta$ may well be reducible), there is no such hope for a similar unique ergodicity result. 
However, the definition of invariant measures on $\Sigma_\zeta$, through
limit frequencies 
as known from the primitive case, 
extends 
naturally via the results of this paper to any expanding reducible substitution $\zeta$. 
We will use the remainder of this subsection to elaborate this, and to comment 
on 
some related developments.

\smallskip

Every shift-invariant measure $\mu$ on $\Sigma_\zeta$ defines a function $\omega_\mu: A^* \to \R_{\geq 0}$ by setting $\omega_\mu(w) := \mu({\rm Cyl}_w)$ if $w$ belongs to $\cal L_\zeta$, and $\omega_\mu(w) := 0$ otherwise.

Conversely, 
it is well known 
(see for instance \cite{FM})
that any function $\omega: A^* \to \R_{\geq 0}$ is defined by an invariant measure $\mu$
on the full shift $A^\Z$
if and only if $\omega$ is a {\em weight function}, i.e. $\omega$ satisfies the Kirchhoff conditions spelled out in Definition \ref{Kirchhoff} below.  In this case $\omega$ determines $\mu$, i.e. there is a unique invariant measure $\mu$ on 
$A^\Z$
that satisfies $\omega = \omega_\mu$. Furthermore, the support of $\mu$ is contained in $\Sigma_\zeta \subset A^\Z$ if and only of $\omega(w) = 0$ for all $w \in A^* \smallsetminus \cal L_\zeta$.

\begin{defn}
\label{Kirchhoff}
A function $\omega: A^* \to \R_{\geq 0}$ satisfies the {\em Kirchhoff conditions} if for any $w \in A^*$ it satisfies:
$$\omega(w) = \sum_{a_i \in A} \omega(a_i w) = \sum_{a_i \in A} \omega(w a_i)$$
\end{defn}

\begin{prop}
\label{limit-frequencies-Kirchhoff}
Let $\zeta: A \to A^*$ be an expanding substitution, raised to a suitable power according to Proposition \ref{theoremB}. Then for any letter $a \in A$ the function
$$\omega_a: A^* \to \R_{\geq 0}, \,\,\, w \mapsto f_w(a)\, ,$$
given by the limit frequencies $f_w(a)$ from Proposition \ref{theoremB}, satisfies the Kirchhoff conditions.
\end{prop}

\begin{proof}
We consider $\zeta^t(a)$ as in Proposition \ref{theoremB} and observe that any occurrence of a word $w$ as factor in $\zeta^t(a)$, unless it is a prefix, together with its preceding letter $a_i$ in $\zeta^t(a)$ gives an occurrence of the factor $a_i w$, and conversely. The analogous statement holds for factors $w a_i$. Hence for every $w \in A^*$ each of the two equalities in Definition \ref{Kirchhoff}, for $\omega(w) := |\zeta^t(a)|_w$, either holds directly, or else it holds up to an additive constant $\pm 1$. Since by the assumption that $\zeta$ is expanding we have $|\zeta^t(a)| \to \infty$ , the Kirchhoff conditions must hold for the limit quotient function $\omega_a = f_w(a) = \lim_{t\to\infty}\frac{|\zeta^{t}(a)|_w}{|\zeta^{t}(a)|}$.
\end{proof}

\begin{rem}
\label{invariant-measure+}
Since for any $a \in A$ and any $w \notin \cal L_\zeta$ the limit frequencies satisfy $f_a(w) = 0$, we obtain directly from Proposition \ref{limit-frequencies-Kirchhoff} that the weight function $\omega_a$ defines an invariant measure $\mu_a$ on $\Sigma_\zeta$. This proves 
Corollary \ref{invariant-measure1} from the Introduction.
\end{rem}



From the definition via limit frequencies it follows immediately that any of the $\mu_a$ is a {\em probability measure}, i.e. 
$\mu_a(\Sigma_\zeta) = 1$.
Contrary to the primitive case, for an expanding substitution $\zeta$ distinct letters $a_i$ of $A$ may well define distinct measures $\mu_{a_i}$ on $\Sigma_\zeta$. However, 
as it happens in the primitive case, distinct $a_i \in A$ may also define the same measure $\mu_{a_i}$. This raises several 
natural 
questions:

\begin{quest}
\label{yet-open}
Let $\zeta$ be an expanding substitution as before. 
\begin{enumerate}
\item
What is the precise condition on letters $a, a' \in A$ such that they define the same measure $\mu_a = \mu_{a'}$ on $\Sigma_\zeta$ ?
\item
Are there invariant measures on $\Sigma_\zeta$ that are not contained in the convex cone $\cal C_\zeta$,
by which we denote the set of
all 
non-negative 
linear combinations of the $\mu_a$ ?
\item
Which of the measures in $\cal C_\zeta$ have the property that in addition to being invariant under the shift operator they are also 
{\em projectively invariant} under application of the substitution $\zeta$ ? By this we mean that
there exist some scalar $\lambda > 0$ such that the image measure $\zeta_*(\mu)$ on $\Sigma_\zeta$ satisfies $\zeta_*(\mu)(X) = \lambda \mu(X)$ for any measurable subset $X \subset \Sigma_\zeta$.
\end{enumerate}
\end{quest}

Attempting seriously to find answers to these questions with the methods laid out here goes beyond the scope of this paper. We limit ourselves to the following:

\begin{rem}
\label{potentially-more}
Our analysis of the eigenvectors of non-negative matrices in PB-Frobenius form in \S \ref{PB-convergence}, when combined with the technique presented in \S \ref{blow-up} above to understand simultaneous eigenvectors for all blow-up level incidence matrices, seems to have the potential to show that the convex cone $\cal C_\zeta$ is spanned by invariant measures that are 
determined 
by the principal eigenvectors (see \S \ref{PB-convergence}) 
of the ``level 1'' incidence matrix $M_\zeta$.
In particular - regarding Question \ref{yet-open} (1) - it seems feasible that $\mu_a = \mu_{a'}$ if and only if $a$ and $a'$ define coordinate vectors $\vec e_a$ and $\vec e_{a'}$ which converge (up to normalization) to the same eigenvector of $M_\zeta$.
\end{rem}

\begin{rem}
\label{big-paper}
In the special case where the substitution $\zeta$, reinterpreted as ``positive'' endomorphism of the free group $F(A)$ with basis $A$, is invertible
with no periodic non-trivial conjugacy classes in $F(A)$, 
a negative answer to Question \ref{yet-open} (2) follows from the the main result of our paper \cite{LU}, 
which was our original motivation to do the work presented here.
\end{rem}

In much more generality reducible substitutions on the whole and 
Question \ref{yet-open} (2) in particular have already been treated in the literature, by the work of Bezuglyi--Kwiatkowski--Medynets--Solomyak, see \cite{BKMS} and the papers cited there. 
A more restricted class of substitutions had been treated previously by Hama--Yuasa, see \cite{HY}.
In particular, the following should be noted:

\begin{rem}
\label{moulinette}
It is shown in \cite{BKMS} 
for expanding substitutions $\zeta$ with a mild extra restriction
that the 
ergodic invariant probability measures on the subshift $\Sigma_\zeta$  
are 
in 1-1 correspondence with the 
normalized
(extremal)
distinguished eigenvectors (see Remark \ref{Schneider})
of the incidence matrix $M_\zeta$
(or 
perhaps 
rather, of the incidence matrix of a conjugate substitution defined there). 

However, 
a direct 
translation of the results of \cite{BKMS}, which is based on Bratteli diagrams and Vershik maps, to the 
framework of the 
work presented here seems 
to be non-evident.
\end{rem}

Also in this context, in particular with respect to Question \ref{yet-open} (3) above, we note:

\begin{rem}
\label{relation-to-BHL}
In the recent preprint \cite{BHL} a conceptually new machinery (called ``train track towers'' and ``weight towers'') 
for subshifts in general 
has been developed, and applied as special case to reducible substitutions $\zeta$ as considered here. As a main result a bijection has been established there between the non-negative eigenvectors of $M_\zeta$ and the ``invariant'' measures on $\Sigma_\zeta$. Although limit frequencies are not treated in \cite{BHL}, it can be seen via weight functions that this bijection is the same as the one indicated in Remark \ref{potentially-more} above.

However, a crucial difference to the work presented here is that in \cite{BHL} ``invariant'' means not just shift-invariance but also projective invariance with respect to the map on measures induced by the substitution $\zeta$, see Question \ref{yet-open} (3) above.
\end{rem}

\section{Primitive Frobenius Matrices and Normalization}
\label{primitive-Frobenius}
\subsection{Normalization functions}
\label{normalization}

Let $\vec U = (\vec u_t)_{t \in \N}$ be an infinite family of vectors in $\R^n$. Then 
$$h 
: \N \to \R$$
is a {\em normalization function} for $\vec U$ if 
\[
\lim_{t\to\infty}\frac{\vec u_t}{h(t)} = \vec v_U
\]
exists, for some limit vector 
$\vec v_{\vec U} \neq \vec 0\, $. 

\begin{rem}
\label{length-normalization}
For any family $\vec U = (\vec u_t)_{t \in \N}$ which possesses a normalization function $h$ as above, the function $h'(t):= \|\vec u_t \|$ is also a normalization function. 
This follows directly from the fact that $\underset{t\to\infty}{\lim}\frac{\vec u_t}{h(t)} = \vec v_U$ implies $\| \vec v_U\| = \underset{t\to\infty}{\lim}{\big \|}\frac{\vec u_t}{h(t)} {\big \|} = \underset{t\to\infty}{\lim}\frac{\|\vec u_t\|}{h(t)} $ and thus $\underset{t\to\infty}{\lim}\frac{\vec u_t}{\|\vec u_t\|} = \underset{t\to\infty}{\lim}\frac{\vec u_t}{h(t)} \frac{h(t)}{\|\vec u_t\|} = \frac{1}{\|\vec v_U\|}\vec v_U$.

\end{rem}

Two functions $f: \N \to \R$ and $g: \N \to \R$ are said to be {\em of the same growth type} if there exist a constant $C > 0$ such that

\[
\lim_{t \to \infty} \frac{f(t)}{g(t)} = C\, .
\]
We say that the growth type of $g$ is strictly bigger than that of $f$ if
\[
\lim_{t \to \infty} \frac{f(t)}{g(t)} = 0\, .
\]

It follows directly 
that, given any infinite family of vectors $\vec U = (\vec u_t)_{t \in \N}$ in $\R^n$, then any two normalization functions $h
$ and $h'
$ for 
$\vec U$ must be of the same growth type, and that, conversely, any other function $h'': \N \to \R$ which is of the same growth type, can be used as normalization function for $\vec U$:
the family 
of values
$\frac{\vec u_t}{h''(t)}$ converges to some non-zero vector in $\R^n$, and 
the latter must be a positive scalar multiple of 
the above limit vector $\vec v_U$.

The following is a direct consequence of the definitions:

\begin{lem}
\label{growth-type-sums}
Let $\vec U = (\vec u_t)_{t \in \N}$ and $\vec U' = (\vec u'_t)_{t \in \N}$ be two infinite family of vectors in $\R^n$, and define $\vec U + \vec U' = (\vec u_t + \vec u'_t)_{t \in \N}$. Let $h 
: \N \to \R$ and $h' 
: \N \to \R$ be normalization functions for $\vec U$ and $\vec U'$ respectively.

\smallskip
\noindent
(1) If the growth type of $h$ is strictly bigger than that of $h'$, then $h$ is also a normalization function for $\vec U + \vec U'$.  Similarly, if the growth type of $h$ is strictly smaller than that of $h'$, then $h'$ is a normalization function for $\vec U + \vec U'$.

\smallskip
\noindent
(2) If $h$ and $h'$ have the same growth type, then a normalization function for $\vec U + \vec U'$ is given by 
both, $h$ or $h'$.
\qed
\end{lem}

\subsection{Lower triangular block matrices}
\label{block-matrices}
Let $M$ be a non-negative integer square matrix. Assume that the rows (and correspondingly the columns) 
of $M$ are partitioned into {\em blocks} $B_i$ so that $M$ is a lower triangular block matrix
with square diagonal matrix blocks. 

We now define a relation on  
the set of 
blocks as follows:
We write
$B_i \succ B_j$ 
if and only if 
$B_i \neq B_j$ and if
there exists a 
non-negative vector $\vec v$ which has non-zero coefficients only in the block $B_i$, such that for some $t\ge1$ the vector $M^{t} \vec v$ has a non-zero coefficient in the block $B_j$.  This is equivalent to stating that for some $t\geq 1$, in the matrix $M^{t}$ the off-diagonal matrix block 
in the $i^{th}$ block column and the $j^{th}$ block row has at least one positive entry. 

For any block $B_i$ we define the {\em dependency block union} $C(B_i)$ to be the union of all blocks $B_j$ with $B_i \succ B_j$. 

Observe that, if every diagonal block of $M$ is either 
irreducible or a $(1\times 1)$-matrix, this relation defines a
partial order on the blocks, 
denoted by writing $B_i \succeq B_j$ if either $B_i = B_j$ or $B_i \succ B_j$.

\smallskip

Let us denote by $\cal C^n$ the non-negative cone in $\R^n$ with respect to the fixed ``standard basis'' $\vec e_1, \ldots ,\vec e_n$. For any block $B_i$ we define the {\em associated cone} $\cal B_i$ as the set of 
all non-negative 
column vectors in $\cal C^n$ that have non-zero entries only in the block $B_i$, i.e. all convex combinations of those $\vec e_i$ that ``belong'' to $B_i$.

A {\em block cone} $\cal C$ is a  
subcone of $\cal C^n$ which has the property that each cone $\cal{B}_i$ is either ``contained or disjoint", i.e. one has either $\cal B_i \subset \cal C$ or $\cal B_i \cap \cal C =  \{\vec 0 \}$. 
Unless otherwise stated, we are only interested in block cones $\cal C$ that are invariant under the action of $M$, i.e.  $M \vec v \in \cal C$ for any $\vec v \in \cal C$. This
is
equivalent to stating that for any block $B$ with $\cal B \subset \cal C$ the block cone $\cal C(B)$ 
(called {\em the dependency block cone}) 
associated to the dependency block union $C(B)$ is contained in $\cal C$.

\subsection{Primitive Frobenius Form}
\label{frobform}

Let $M$ be a non-negative integer square matrix as considered above, 
and assume that $M$ is partitioned into matrix blocks 
so that $M$ is a lower triangular block matrix, and
along the diagonal all matrix blocks are squares.

\begin{defn}
\label{primitive-F}
(1)
The matrix $M$ is said to be in {\em primitive Frobenius form} if
every diagonal matrix block is 
primitive,
including the case of  a $(1 \times 1)$-matrix with entry $1$ or $0$.

\noindent(2)
For 
every  
block $B_i$ we 
refer to 
the Perron-Frobenius eigenvalue $\lambda_i$ 
of the corresponding diagonal block of $M$
as the {\em PF-eigenvalue} 
of 
the block $B_i$.
This includes (for the special case of  a $(1 \times 1)$-zero block $B_i$) the possible value $\lambda_i = 0$. 

\smallskip
\noindent
(3)
For every 
diagonal block $M_{i, i}$ 
of $M$ we define the {\em extended PF-eigenvector} $\vec v = \sum a_i \vec e_i$ to be obtained from a Perron-Frobenius eigenvector of $M_{i, i}$ through adding $0$ as values in all other coordinates, subject to the condition that 
$\|\vec v\| = \sum |a_i| = 1$.
\end{defn}

\begin{rem}
\label{iterating-to-get-primitive-F}
(a) Every non-negative integer square matrix $M$ has a positive power which is in primitive Frobenius form. This is a direct consequence of the well known normal form for non-negative matrices (compare the proof of Lemma \ref{Frobenius-powers*}).

\smallskip
\noindent
(b) If $M$ is already in PB-Frobenius form (see Definition \ref{Frobenius-form*}), and the positive power $M^t$ is in primitive Frobenius from, then it follows directly from the definitions that the block decomposition for $M$ agrees outside of the PB-blocks with that of $M^t$, while the PB-blocks for $M$ need possibly be partitioned further to get the blocks for $M^t$.
\end{rem}

For any matrix $M$ in primitive Frobenius form we define the {\em growth type} associated to any of its blocks $B_i$ as follows: 
Among the blocks $B_j$ with $B_j \preceq B_i$, we consider the maximal PF-eigenvalue $\lambda_{max}(B_i) := \max\{\lambda_j \mid B_i \succeq B_j\}$, and the longest 
(or rather: ``a longest'') 
chain of blocks $B_{i_k} \succ B_{i_{k-1}} \succ \ldots \succ B_{i_1}$ which all have PF-eigenvalue $\lambda_{i_k} = \lambda_{i_{k-1}} = \ldots = \lambda_{i_1}=\lambda_{max}(B_i)$.
We then define the function 
$$h_i: \N \to \R, \, t \mapsto \lambda_{max}(B_i)^t \cdot t^{k-1}$$ 
as the {\em growth type function} of the block $B_i$.

Similarly, we define the {\em growth type function} 
$h_C: \N \to \R$
of any union of blocks $C$  (or of the associated block cone $\cal C$) as the maximal growth type function $h_j$ of any $B_j$ which belongs to $C$.

\begin{defn}[Dominant Interior]
\label{domin} 
Let  $\cal C$ be the block cone associated to any 
union $C$ of blocks. 
Define the {\em dominant interior} of $\cal{C}$
as follows:
Pick some longest chain of blocks $B_{i_k} \succ B_{i_{k-1}} \succ \ldots \succ B_{i_1}$ as above, i.e. all $B_{i_j}$ have PF-eigenvalue $\lambda_{i_k} = \lambda_{i_{k-1}} = \ldots = \lambda_{i_1}=\lambda_{max}(C)$
(in other words: the block $B_{i_j}$ is part of a ``realization'' of the growth type function $h_C$).

Let $\vec v \in \cal C$ be a vector for which the coordinates, for all vectors $\vec e_i$ of the standard basis that belong to one of the blocks $B_{i_j}$, are non-zero. 
The dominant interior of $\cal C$ consists of all such vectors $\vec v$, for any longest chain of blocks as above, 
which may of course vary with the choice of $\vec v$. 
\end{defn}

\begin{lem}
\label{elementary-fact}
Let $M$ be 
in primitive Frobenius form.
Then there exists a bound $t_0 \geq 0$ such that for any 
blocks $B_i$ and $B_j$ of $M$ with 
$i < j$ 
and for any exponent 
$t \geq t_0$ the power $M^t$ of $M$ has 
off-diagonal block $M_
{j,i}
^{t}$ 
which is 
positive (i.e. has only positive coefficients) if $B_i \succ B_j$ and if 
one of the diagonal blocks $M_{i,i}$ or 
$M_{j,j}$ of $M$ is primitive non-zero. Otherwise $M_
{j,i}^t$ is zero.
\end{lem}

\begin{proof}
We can first rise $M$ to a positive power $M^s$ such that any primitive non-zero diagonal block $M_{i,i}^s$ of $M^s$ is positive.
It follows that the same is true for any exponent $s' \geq s$.

If $B_i \succ B_j$, then by definition of $\succ$ for some integer $k = k(i,j)$ the power $M^k$ has in its off-diagonal block $M_
{j,i}^k$ some positive coefficient $a_{p,q}$. If both, $M_{i,i}$ and $M_{j,j}$ are primitive non-zero, it follows that for $M^{k+2s}$ the same diagonal block is positive, and this is also true for any exponent $t \geq k + 2s$.

If $B_i \succ B_j$ and $M_{i,i}$ is primitive non-zero but $M_{j,j}$ is zero, we deduce from above the positive coefficient $a_{p,q}$ of $M_
{j,i}^k$ that for $M^{k+s}$ all coefficients in the $p$-th line of the block $M_
{j,i}^{k+s}$ must be positive. We now use the fact that the diagonal zero matrix $M_{j,j}$ must be a $(1\times1)$-matrix, so that $M_
{j,i}^{k+s}$ consists of a single line, which is thus positive throughout. The same argument holds for any $t = k+s'$ with $s' \geq s$.

If $B_i \succ B_j$ and $M_{j,j}$ is primitive non-zero but $M_{i,i}$ is zero, we deduce from above the positive coefficient $a_{p,q}$ of $M_
{j,i}^k$ that for $M^{k+s}$ all coefficients in the $q$-th column of the block $M_
{j,i}^{k+s}$ must be positive. We now use the fact that the diagonal zero matrix $M_{i,i}$ must be a $(1\times1)$-matrix, so that $M_
{j,i}^{k+s}$ consists of a single column, which is thus positive throughout. The same argument holds for any $t = k+s'$ with $s' \geq s$.

If $B_i \succ B_j$ and both, $M_{i,i}$ and $M_{j,j}$ are zero matrices, then $M^2$ has as 
$(j,i)$-th block the zero matrix, and the same is true for all powers $M^t$ with $t \geq 2$. 

Finally, 
if it doesn't hold that $B_i \succ B_j$, then by definition of $\succ$ the 
$(j,i)$-th block of any positive power of $M$ is the zero matrix.  
\end{proof}

\begin{lem}
\label{initial-faces}
Let 
$M$ be in primitive Frobenius form, 
and assume that there are no zero columns in $M$.
Let $B_i$ be any 
block of $M$, 
let $C_i = C(B_i)$ be its dependency block union,
and
let $\cal B_i$ and $\cal C_i$ be the corresponding block cones. For any non-zero vector $\vec v \in \cal B_i$ we write 
$$M^t \vec v = \vec v^*_t + \vec u^*_t$$
with $
\vec v^*_t \in \cal B_i$ and $
\vec u^*_t \in \cal C_i$. 

Then there is a bound $t_0 \in \N$ depending only on $M$ such that for every $t \geq t_0$ the vector $\vec u^*_t$ is contained in the dominant interior of $\cal C_i$. 
\end{lem}

\begin{proof}
Let $t_0$ be as in Lemma \ref{elementary-fact}. Then $\vec v^*_t + \vec u^*_t = M^t v$ has positive coordinates in all blocks $B_j$ of $C_i$ for which $M$ has a primitive non-zero diagonal block $M_{j,j}$. Since $M$ has no zero-columns, the maximal eigenvalue for the blocks in $C_i$ must be strictly bigger then $0$. Thus the dominant interior of $C_i$ is defined through chains of blocks which are primitive non-zero. Hence $\vec u^*_t$ is contained in the dominant interior.
\end{proof}

\subsection{An example}
Before proceeding with the proof of the main theorem, we discuss an example explaining the above concepts: 

Let $M$ be the following matrix: 
\[
M=
  \begin{bmatrix}
     3& 1& 0 & 0 & 0  & 0 &0 & 0\\
 1& 1& 0 & 0 & 0  & 0 &0 & 0 \\
 1 & 2 & 2 & 1 & 0 & 0 & 0 & 0\\
 1 & 1 & 1 & 1 & 0  & 0 &0 &0\\
 4 & 0 & 0 & 0 & 3 & 1& 0 & 0\\
 1 & 1 & 0 & 0 & 1  & 1 & 0 & 0\\ 
  0 & 3 & 1 & 3& 2& 3& 2 & 1\\
 1 & 1 & 2 & 1 & 0  & 4 & 1 & 1\\ 
  \end{bmatrix}
\]

The matrix $M$ is partitioned into 4 blocks $B_1, B_2, B_3, B_4$, where $M_{1,1}=  \begin{bmatrix} 3 & 1\\ 1  & 1\\  \end{bmatrix}$,  $M_{2,2}=  \begin{bmatrix} 2 & 1\\ 1  & 1\\  \end{bmatrix}$, $M_{3,3}=  \begin{bmatrix} 3 & 1\\ 1  & 1\\  \end{bmatrix}$ and $M_{4,4}=  \begin{bmatrix} 2 & 1\\ 1  & 1\\  \end{bmatrix}$

\medskip
We have the following relations: $B_1\succ B_2\succ B_4$, $B_1\succ B_3$, $B_3\succ B_4$. We compute that $PF(B_1)=PF(B_3)=2+\sqrt{2}$ and  $PF(B_2)=PF(B_4)=
\frac{\sqrt{5}+3}{2} < 2+\sqrt{2}$. 

Hence, with the above definitions $B_1$ has growth type $t(2+\sqrt{2})^{t}$,  $B_2$ has $t(\frac{\sqrt{5}+3}{2})^{t}$, $B_3$ has $(2+\sqrt{2})^{t}$, and $B_4$ has $(\frac{\sqrt{5}+3}{2})^{t}$. 

The dependency blocks are given by $C(B_1) = B_2 \cup B_3 \cup B_4$, $C(B_2) = C(B_3) = B_4$, and $C(B_4) = \emptyset$.

The dominant interiors are given (where $\overset{\circ}{X}$ denotes the interior of a space $X$)

\noindent
for\,\,\,  $\cal B_1 + \cal C(B_1)$ \,\,\,  by \,\,\,  
$\overset{\circ}{\cal B_1} + \overset{\circ}{\cal B_3} + \cal B_2 + \cal B_4$,
\noindent
for\,\,\,  $\cal B_2 + \cal C(B_2)$ \,\,\,  by \,\,\,  
$\overset{\circ}{\cal B_2} + \overset{\circ}{\cal B_4}$,

\noindent
for\,\,\,  $\cal B_3 + \cal C(B_3)$ \,\,\,  by \,\,\,  
$\overset{\circ}{\cal B_3} + \cal B_4$,

\noindent
for\,\,\,  $\cal B_4 + \cal C(B_4)$ \,\,\,  by \,\,\,  
$\overset{\circ}{\cal B_4}$.

There is one more $M$-invariant block cone, given by $\cal C = \cal B_2 + \cal B_3 + \cal B_4$. Its dominant interior is given by $
\overset{\circ}{\cal B_3} + \cal B_2 + \cal B_4$
\section{Convergence for primitive Frobenius matrices}
\label{section-proof}

The goal of this and the following 
section is to  give a complete proof of the following result.
For related statements the reader is directed to the work of H. Schneider \cite{Sc} and the references given there.

\begin{thm}\label{GPFT}
Let $M$ be a non-negative integer square matrix
which is in primitive Frobenius form as  
given in Definition \ref{primitive-F}.
Assume that $M$ has no zero columns. 
Then for any non-negative vector $\vec{v} \neq \vec 0$ there exists a normalization function $h_{\vec v}$ such that 
\[
\lim_{t\to\infty}\frac{M^{t}\vec{v}}{h_{\vec v}(t)}=\vec{v}_{\infty}\, ,
\]
where $\vec{v}_{\infty} \neq \vec 0$ is an 
eigenvector of $M$. 
\end{thm}

This result is proved by induction, and the induction step has some interesting features in itself, so that we pose it here as independent statement. But first we state a property which will be used below repeatedly:

\begin{defn}
\label{convergence-cond}
Let $M$ be as in Theorem \ref{GPFT}, and 
let $C$ be a union of matrix blocks such that the associated block cone $\cal C \subset \cal C^n$ is $M$-invariant, with growth type function $h_C = \lambda_*^t t^{d_*}$ for some value $\lambda_* \geq 1$. We say that $\cal C$ satisfies the {\em convergence condition} CC($\cal C$) if
%
for every vector $\vec u \in \cal C$ 
the sequence $\frac{1}{h_C(t)}M^t \vec u$ converges to a vector $\vec u_\infty$ which is either an 
eigenvector $\vec u_\infty \in \cal C$ of $M$, or else one has $\vec u_\infty = \vec 0$. We require furthermore that $\vec u_\infty \neq \vec 0$ if $\vec u$ is contained in the dominant interior of $\cal C$ (as defined 
above in Definition \ref{domin}).
\end{defn}

\begin{rem}
\label{eigenvalue}
(a)
For $\vec u_\infty$ as in Definition \ref{convergence-cond} the condition $\vec u_\infty \neq \vec 0$ implies  
directly 
that $\vec u_\infty$ is an eigenvector of $M$. 

\smallskip
\noindent
(b)
Its eigenvalue is always equal to $\lambda_*$, as follows directly from the following consideration:
$$M \vec u_\infty = M (\lim_{t \to \infty} \frac{1}{h_C(t)}M^t \vec u) =  (\lim_{t \to \infty} \frac{h_C(t+1)}{h_C(t)}\frac{1}{h_C(t+1)}M^{t+1} \vec u) =$$
$$ (\lim_{t \to \infty} \frac{h_C(t+1)}{h_C(t)}) \vec u_\infty = (\lim_{t \to \infty} \frac{\lambda_*^{t+1} (t+1)^{d_*}}{\lambda_*^{t} (t)^{d_*}}) \vec u_\infty = \lambda_*\vec u_\infty$$
\end{rem}

\begin{prop}
\label{induction-step}
Let $M$ be a non-negative integer square matrix
which is in primitive Frobenius form,
with no zero columns. 
Let $B$ be any block of the associated block decomposition, and let $C := C(B)$
be the corresponding dependency block union
(see \S \ref{normalization}).
Let $\cal B$ and $\cal C$ be the block cones associated to $B$ and $C$ respectively.

Let $\lambda \geq 0$ 
 and $\lambda_u \ge 1$ be the maximal PF-eigenvalues of $B$ and $C$ respectively, and let $h: t \mapsto \lambda_*^t t^{d}$ 
(for $\lambda_* = \max\{\lambda, \lambda_u\}$) and $h_u: t \mapsto \lambda_u^t t^{d_u}$
be the growth type functions for $B$ and $C$ respectively (see \S\ref{frobform}).

Assume that  $\cal C$ satisfies the above convergence condition CC($\cal C$). Then for every vector $\vec 0 \neq \vec v_0 \in \cal B$ the sequence 
$$\vec v_t := \frac{1}{h(t)} M^t \vec v_0$$
converges to an eigenvector $\vec w_\infty$
of $M$ which satisfies:
\begin{enumerate}
\item 
If $\lambda > \lambda_u$ then $\vec w_\infty = \lambda(\vec v_0) (\vec v_\infty + \vec w_0)$, where $\vec v_\infty$ is the extended PF-eigenvector (see Definition \ref{primitive-F} (3)) of the primitive diagonal block of $M$ corresponding to $B$,
the vector $\vec w_0 \in \cal C$ is entirely determined by $\vec v_\infty$, and $\lambda(\vec v_0) \in \R_{> 0}$ depends on $\vec v_0$.
\item
If $\lambda = \lambda_u$ then $\vec w_\infty = \lambda(\vec v_0) \vec u_\infty$, where $\vec u_\infty \neq \vec 0$ is an eigenvector of $\cal C$ that depends only on the above 
extended PF-eigenvector $\vec v_\infty$, 
and $\lambda(\vec v_0) \in \R_{> 0}$ depends on $\vec v_0$.
\item
If $\lambda < \lambda_u$ then $\vec w_\infty \neq \vec 0$ is an eigenvector of $\cal C$ that may well depend on the choice of $\vec v_0$.
\end{enumerate}
\end{prop}

Before proving Proposition \ref{induction-step} in section \ref{theproof}, we first show how to derive Theorem \ref{GPFT} from Proposition \ref{induction-step}.  We first show that Proposition \ref{induction-step} also implies the following:

\begin{lem}
\label{weak-interior}
Assume that $B$ and $C$ as well as $\cal B$ and $\cal C$ are as in Proposition \ref{induction-step}.
Then we have:

\smallskip
\noindent
(1) 
The cone $\cal B + \cal C$ associated to the block union $B \cup C$ satisfies the convergency condition {\rm CC($\cal B + \cal C$)}.

\smallskip
\noindent
(2)
Assume that $\cal C$ is contained in a larger block cone $\cal C'$ with growth type function $h'$, and assume that 
$\cal C'$ satisfies the convergency condition {\rm CC($\cal C'$)}. Then the cone $\cal B +\cal C'$ also satisfies the convergency condition {\rm CC($\cal B + \cal C'$)}.
\end{lem}

\begin{proof}
(1)
If $B$ belongs to the blocks of $B \cup C$ that determine the dominant interior of $\cal B + \cal C$, then the 
eigenvalue of the PF-eigenvector of $B$ satisfies $\lambda \geq \lambda_u 
\geq 
1$, and 
is maximal among all PF-eigenvalues for blocks in $B \cup C$. 
If $\lambda>\lambda_{u}$, then 
\[
\lim_{t\to\infty}\frac{h_u(t)}{h(t)} = \lim_{t\to\infty}\frac{\lambda_u^t t^{d_u}}{\lambda^t} = 0.\]
If $\lambda=\lambda_{u}$ and hence $d=d_u+1$, we have 
\[
\lim_{t\to\infty} \frac{h_u(t)}{h(t)} = \lim_{t\to\infty}\frac{\lambda_u^t t^{d_u}}{\lambda^t t^d} = \lim_{t\to\infty} \frac{1}{t} = 0.
\]

We note that 
case (3) of Proposition \ref{induction-step} is excluded by the inequalities 
$\lambda \geq \lambda_u$,
and that 
in cases (1) and (2) of Proposition \ref{induction-step} our claim $\underset{t\to\infty}{\lim} \frac{1}{h(t)} M^t \vec v \neq \vec 0$ is explicitly stated for any non-zero $\vec v \in \cal B$. 
For arbitrary $\vec v$ in the dominant interior of $\cal B + \cal C$ we conclude the claim from $\underset{t\to\infty}{\lim} \frac{h_u(t)}{h(t)} = 0$ and from Lemma \ref{growth-type-sums} (1).

If $B$ does not belong to the blocks of $B \cup C$ that determine the dominant interior, then we have $\lambda_u > \lambda$, so that we are in case (3) of Proposition \ref{induction-step}:
In this case, however, any vector in the dominant interior of $\cal B + \cal C$ must also belong to
the dominant interior of $\cal C$. 
The growth type function for $B \cup C$ is given by $h = h_u$, and 
hence 
the claim follows from our assumption 
CC($\cal C$). 

\smallskip
\noindent
(2)
Similar to the situation considered above in the proof of (1), if $B$ does not belong to the blocks that determine the dominant interior of $\cal B + \cal C'$, then any vector in the dominant interior of $\cal B + \cal C'$ must also belong to the dominant interior of $\cal C'$, and the growth type function for $\cal B + \cal C'$ is equal to that for $\cal C'$, so that the claim follows from the assumption
CC($\cal C'$). 

If on the other hand $B$ belongs to the blocks that determine the dominant interior of $\cal B + \cal C'$, then the growth type function for $\cal B + \cal C'$ is equal to that of $B$, so that part (1) shows that the limit vector is non-zero for any $\vec v \neq \vec 0$ in the dominant interior of $\cal B + \cal C$. Any vector $\vec w$ in the dominant interior of $\cal B + \cal C'$ can be written as sum
$\vec w = \vec v + \vec u + \vec w_0$ 
where
$\vec w_0$ belongs to $\cal B + \cal C'$ but not to its dominant interior, while
$\vec v$ lies in the dominant interior of $\cal B + \cal C$ and $\vec u$ in the dominant interior of $\cal C'$, and at least one of them is non-zero.
Thus we deduce
the claim follows directly from Lemma \ref{growth-type-sums}, applied to $\vec v$ and $\vec u$.

\end{proof}

We will now prove Theorem \ref{GPFT}, assuming the results of Proposition \ref{induction-step}.
The proof of Proposition \ref{induction-step} is deferred to section \ref{theproof}. 

\begin{proof}[Proof of Theorem \ref{GPFT}] 
Consider the block decomposition of $M$ according to its primitive Frobenius form, and denote by $B$ the top matrix block. Let $C = C(B)$ be the corresponding dependency block union.

If $C$ is empty, then $B$ is minimal with respect to the partial order on blocks
(as defined in subsection \ref{block-matrices}).
In this case, from the assumption that $M$ has no zero columns, it follows that $B$ is not a zero matrix. Hence the claim of Theorem \ref{GPFT} 
for any vector $\vec v \in \cal B$ 
follows directly from the classical Perron-Frobenius theory.

If $C$ is non-empty, 
it follows from the previously considered case that the maximal eigenvalue for $C$ satisfies $\lambda_u \geq 1$. 
Thus via induction over the number of blocks contained in $C$
we can invoke Lemma \ref{weak-interior} (2) 
to obtain that 
the convergency condition CC($\cal C$) holds.

We can hence apply Proposition \ref{induction-step} to get directly the 
the claim of Theorem \ref{GPFT} 
for any non-negative vector $\vec v \in \cal B$.

We can then assume by induction that the claim of Theorem \ref{GPFT} is true for any vector $\vec u \neq \vec 0$ that has zero-coefficients in the $B$-coordinates. 
Now, an arbitrary vector $\vec w \neq \vec 0$ in the non-negative cone $\cal{C}^n$ 
can be written 
as a sum $\vec w = \vec v + \vec u$, with $\vec v$ and $\vec u$ as before, and at least one of them 
is 
different from $\vec 0$. Hence the claim of Theorem \ref{GPFT} follows from Lemma \ref{growth-type-sums}.
\end{proof}

\begin{rem}
The last proof also shows the following slight improvement of Theorem \ref{GPFT}:
For every primitive block $B_i$ of the Frobenius form of $M$, and for any vector $\vec v \neq \vec 0$ in the associated non-negative cone $\cal B_i$, the normalization function 
$h_{\vec v}$ from Theorem \ref{GPFT}
for the family $(M^t \vec v)_{t \in \N}$  is of the same growth type as the function $h_i$ defined in section \ref{frobform}.
\end{rem}

Recall from section \ref{prelims*} that
\[
{\big \|}\sum a_i \vec e_i {\big \|} = \sum |a_i|\, . 
\]
The following 
elementary observation is repeatedly used in the next section.

\begin{lem}
\label{universal-constant}
Let $M$ be a non-negative integer $(n \times n)$-matrix. Assume that there exists a 
function $h: \N \to \R_{>0}$ such that for any vector $\vec u$ in the non-negative cone $\cal C^n = (\R_{\geq 0})^n$
the sequence
\[
\frac{1}{h(t)}M^t \vec u
\]
converges to a limit vector $\vec u_\infty \in 
\cal C^n$ 
which is either equal to $\vec 0$ or else an eigenvector $\vec u_\infty \in 
\cal C^n$ 
of $M$.




Then there is a ``universal constant'' $K = K(\cal C) > 0$ which satisfies:
$$
\frac{1}{h(t)}\frac{||M^t \vec v||}{||\vec v||} \leq K 
$$
for any $t\in\N$ and for any (not necessarily non-negative) $\vec v \in \R^n$. \end{lem}

\begin{proof}
We first consider the finitely many coordinate vectors $\vec e_i$ from the canonical base of $\R^n$ and observe that the hypothesis 
$$
\lim_{t \to \infty}\frac{1}{h_(t)}M^t \vec e_i = \vec u_\infty^i
$$
for some $\vec u_\infty^i \in \cal C$ implies the existence of a constant $K_0 > 0$ with
$$
\frac{1}{h_(t)}\frac{||M^t \vec e_i||}{||\vec e_i||} \leq K_0 
$$
for any $t \in \N$ and any $i = 1, \ldots, n$. 

An arbitrary vector $\vec v = \sum a_i \vec e_i \in \R^n$ satisfies $||\vec v|| = \sum |a_i| \geq |a_i|\cdot ||\vec e_i||$, which gives
$$
\frac{1}{h(t)}\frac{||M^t \vec v||}{||\vec v||} = 
\frac{1}{h(t)}\frac{||\sum_{i = 1}^n a_i M^t \vec e_i||}{||\vec v||} \leq 
\frac{1}{h(t)}\sum_{i = 1}^n \frac{|a_i|\cdot||  M^t \vec e_i||}{||\vec v||} \leq
$$
$$
\frac{1}{h(t)}\sum_{i = 1}^n \frac{|a_i|\cdot||  M^t \vec e_i||}{|a_i|\cdot ||\vec e_i||} \leq 
\sum_{i = 1}^n \frac{1}{h(t)}\frac{|  M^t \vec e_i||}{||\vec e_i||} \leq 
 n K_0 \, ,
$$
thus proving the claim for $K(\cal C) := n K_0$.
\end{proof}

\section{Proof of the Proposition \ref{induction-step}}
\label{theproof}

Let us consider an arbitrary vector $\vec 0 \neq \vec v_0 \in \cal B$, and define iteratively, for any integer $t \geq 1$, vectors $v_t \in \cal B$ and $u_t \in \cal C$ through 
\[
M \vec  v_{t-1} = \lambda \vec v_t + \vec u_t.
\]
Therefore, for any $t \geq 1$, we compute
\[
M^t \vec v_0
 = \lambda^t \vec v_t + \sum_{k=0}^{t-1} \lambda^{k} M^{t-k-1} \vec u_{k+1}
 = \lambda^t \vec v_t + \sum_{m=0}^{t-1} \lambda^{t-m-1} M^{m} \vec u_{t-m}.
\]

\medskip
\noindent
{\bf Case 1:} Assume that $\lambda_u < \lambda$.  
\smallskip

In this case  
the diagonal block $M_{ii}$ of $M$ corresponding to $B$
is primitive. Let $\vec{v} \in \cal B$ be 
the extended PF-eigenvector of $M$ as given in section \ref{frobform}.
 
Let $\vec u \in \cal C$ be the non-negative vector determined by the equation
\[
M \vec v = \lambda \vec v + \vec{u} \, .
\]
Then we compute:
\begin{align*}
\frac{1}{\lambda^t} M^t \vec v
&= \frac{1}{\lambda^t} (\lambda^t \vec v + \sum_{m=0}^{t-1} \lambda^{t-m-1} M^{m} \vec u)\\
&= \vec v +  \frac{1}{\lambda}\sum_{m=0}^{t-1} 
 \frac{\lambda_u^{m} \cdot {m}^{d_u}}{\lambda^{m}}
 \frac{1}{\lambda_u^{m} \cdot {m}^{d_u}} M^{m} \vec u
\end{align*}
Recall that, 
since $\vec u \in \cal C$, by assumption there is a vector $\vec u_\infty \in \cal C$ with
\[
\lim_{m\to\infty}\frac{1}{\lambda_u^{m} \cdot {m}^{d_u}} M^{m} \vec u=\vec{u}_{\infty}.
\]
Hence we deduce that for some constant $K\ge0$ one has
\[
\|\frac{1}{\lambda_u^{m} \cdot {m}^{d_u}} M^{m} \vec{u}\| \leq K\]
for all $m\ge1$. From here it follows that the series
\[
\sum_{m=0}^{\infty}\frac{\lambda_{u}^m\cdot m^{d_u}}{\lambda^m}\frac{1}{\lambda_{u}^m\cdot m^{d_u}}M^m\vec{u}=\sum_{m=0}^{\infty}\left(\frac{\lambda_u}{\lambda}\right)^m\cdot m^{d_u}\frac{1}{\lambda_{u}^m\cdot m^{d_u}}M^m\vec{u}
\]
is convergent. 
Set
\[
\vec w :=\sum_{m=0}^{\infty} 
 \frac{\lambda_u^{m} \cdot {m}^{d_u}}{\lambda^{m}}
 \frac{1}{\lambda_u^{m} \cdot {m}^{d_u}} M^{m} \vec u=\sum_{m=0}^{\infty}\frac{1}{\lambda^m}M^m\vec{u}.
 \]

We now observe:
\begin{align*}
\frac{1}{\lambda} M (\vec v + \frac{1}{\lambda} \vec w)&=\frac{1}{\lambda}\left(\lambda\vec{v}+\vec{u}+\frac{1}{\lambda}\sum_{m=0}^{\infty} 
\frac{1}{\lambda^{m}} M^{m+1} \vec u\right)
\\&=\frac{1}{\lambda}\left(\lambda\vec{v}+\vec{u}+\sum_{m=1}^{\infty} 
\frac{1}{\lambda^{m}} M^{m} \vec u\right)
\\&=\frac{1}{\lambda}\left(\lambda\vec{v}+\vec{u}+\vec{w}-\vec{u}\right) 
\\&=\vec v + \frac{1}{\lambda} \vec{w}
\end{align*}
In other words, $\vec v + \frac{1}{\lambda} \vec w$ is an eigenvector of $M$ with eigenvalue $\lambda$ which is contained in the non-negative cone $\cal B + \cal C$ spanned by $B$ and $C$. 

\medskip

We now consider an arbitrary vector $\vec v_0 \in \cal B$, as well as the vectors  $v_t \in \cal B$ and $u_t \in \cal C$ as defined iteratively 
at the beginning of this section.
For any integer $s$ with $1 \leq s \leq t-1$, we have 
\begin{align*}
\frac{1}{\lambda^t} M^t \vec v_0
&= \frac{1}{\lambda^t} (\lambda^t \vec v_t + \sum_{m=0}^{t-1} \lambda^{t-m-1} M^{m} \vec u_{t-m})
\\&= \vec v_t +  \frac{1}{\lambda}\sum_{m=0}^{t-1} 
 \frac{\lambda_u^{m} \cdot {m}^{d_u}}{\lambda^{m}}
 \frac{1}{\lambda_u^{m} \cdot {m}^{d_u}} M^{m} \vec u_{t-m}
 \\&= \vec v_t 
 + \frac{1}{\lambda}\sum_{m=0}^{s} 
 \frac{\lambda_u^{m} \cdot {m}^{d_u}}{\lambda^{m}}
 \frac{1}{\lambda_u^{m} \cdot {m}^{d_u}} M^{m} \vec u_{t-m}
 + \frac{1}{\lambda}\sum_{m=s+1}^{t-1} 
 \frac{\lambda_u^{m} \cdot {m}^{d_u}}{\lambda^{m}}
 \frac{1}{\lambda_u^{m} \cdot {m}^{d_u}} M^{m} \vec u_{t-m}
 \\
 &\tag{$\dagger$}\label{firstsum}= \vec v_t 
 + \frac{1}{\lambda}\sum_{m=0}^{s} 
 \frac{\lambda_u^{m} \cdot {m}^{d_u}}{\lambda^{m}}
 \frac{1}{\lambda_u^{m} \cdot {m}^{d_u}} M^{m} \vec u_{t-m}
 + (\frac{\lambda_u}{\lambda})^{s}
\frac{1}{\lambda}\sum_{m=s+1}^{t-1} 
 \frac{\lambda_u^{m-s} \cdot {m}^{d_u}}{\lambda^{m-s}}
 \frac{1}{\lambda_u^{m} \cdot {m}^{d_u}} M^{m} \vec u_{t-m} 
\end{align*}

We now consider the limit of this sum for $t \to \infty$: By the classical Perron-Frobenius theorem for primitive non-negative matrices we have 
\[
\lim_{t\to\infty} \vec v_t = \lambda' \vec v
\] for some $\lambda' > 0$. 
From our definition of the $\vec v_t$ and $\vec u_t$ it follows that their lengths 
$\|\vec v_t\|$ and $\|\vec u_t\|$
are uniformly bounded. We can hence apply Lemma \ref{universal-constant} 
to the subspace $\R^m \subset \R^n$ generated by $\cal C$ in order
to deduce that 
there is a uniform bound to the length of any of the $\frac{1}{\lambda_u^{m} \cdot {m}^{d_u}} M^{m} \vec u_{t-m}$. Hence for any $s \geq 0$ the sum
\[
\sum_{m=s+1}^{t-1} 
 \frac{\lambda_u^{m-s} \cdot {m}^{d_u}}{\lambda^{m-s}}
 \frac{1}{\lambda_u^{m} \cdot {m}^{d_u}} M^{m} \vec u_{t-m}
\]
converges for $t \to \infty$. As a consequence, for any $\epsilon > 0$ there is a value $s = s(\epsilon) \geq 0$ such that for any $t \geq s+2$ the third term of the above sum (\ref{firstsum}) satisfies: 
\[
\left\|(\frac{\lambda_u}{\lambda})^{s}
\ss
\frac{1}{\lambda}
\sum_{m=s+1}^{t-1} 
 \frac{\lambda_u^{m-s} \cdot {m}^{d_u}}{\lambda^{m-s}}
 \frac{1}{\lambda_u^{m} \cdot {m}^{d_u}} M^{m} \vec u_{t-m}\right\|
 \leq \epsilon.
\]

On the other hand, for large values of $t$ the vectors $\vec v_{t-m-1}$ will be close to $\lambda' \vec v$, and hence $\vec u_{t-m}$ will be close to $\lambda' \vec u$, for $\vec u$ as defined above by means of the eigenvector $\vec v$. That is,
for any $\epsilon > 0$ there is a bound $t_0 = t_0(\epsilon) \geq 0$ such that for any $t \geq t_0$ there is a (not necessarily non-negative !) vector $\vec w_t$ of length 
\[
\|\vec w_t\| \leq \epsilon\]
with $\vec u_{t} = \lambda' \vec u + \vec w_{t}$. This gives, for any $s \leq t - t_0\,$:
\begin{align*} 
&\frac{1}{\lambda}\sum_{m=0}^{s} 
 \frac{\lambda_u^{m} \cdot {m}^{d_u}}{\lambda^{m}}
 \frac{1}{\lambda_u^{m} \cdot {m}^{d_u}} M^{m} \vec u_{t-m}
=\\&
\frac{\lambda'}{\lambda}\sum_{m=0}^{s} 
 \frac{\lambda_u^{m} \cdot {m}^{d_u}}{\lambda^{m}}
 \frac{1}{\lambda_u^{m} \cdot {m}^{d_u}} M^{m} \vec u
+
\frac{1}{\lambda}\sum_{m=0}^{s} 
 \frac{\lambda_u^{m} \cdot {m}^{d_u}}{\lambda^{m}}
 \frac{1}{\lambda_u^{m} \cdot {m}^{d_u}} M^{m} \vec w_{t-m}.
\end{align*}

We compute
\begin{align*}
&\left\|\frac{1}{\lambda}\sum_{m=0}^{s} 
\frac{\lambda_u^{m} \cdot {m}^{d_u}}{\lambda^{m}} 
\frac{1}{\lambda_u^{m} \cdot {m}^{d_u}} M^{m} \vec w_{t-m}\right\|
\\& 
\leq \frac{1}{\lambda}\sum_{m=0}^{s} 
\frac{\lambda_u^{m} \cdot {m}^{d_u}}{\lambda^{m}} 
\left\|\frac{1}{\lambda_u^{m} \cdot {m}^{d_u}} M^{m} \vec w_{t-m}\right\|
\\& 
\leq \frac{1}{\lambda}\sum_{m=0}^{s} 
\frac{\lambda_u^{m} \cdot {m}^{d_u}}{\lambda^{m}} 
K(\epsilon),
\end{align*}
where $K(\epsilon)$ is the constant from Lemma \ref{universal-constant} 
(again applied to the subspace generated by $\cal C$). 
As a consequence, for 
any $t \geq s + t_0(\epsilon)$
and some constant $K'$ which only depends on $\cal C$
the second term in the above sum \label{firstsum}
will be $\epsilon K'$-close to 
\[
\frac{\lambda'}{\lambda}\sum_{m=0}^{s} 
 \frac{\lambda_u^{m} \cdot {m}^{d_u}}{\lambda^{m}}
 \frac{1}{\lambda_u^{m} \cdot {m}^{d_u}} M^{m} \vec u\, ,
\]
which converges (according to the above definition of $\vec w$) 
to $\dfrac{\lambda'}{\lambda} \vec w$ as $s$ tends to infinity. 

Given $\epsilon > 0$, use the first part of our considerations to find $s = s(\epsilon)$ which ensures that the third term in the above sum \ref{firstsum} is smaller than $\epsilon$. We then find $t_0 = t_0(\frac{\epsilon}{K'})$, and consider any value $t \geq t_0 + s$. The above derived estimates 
give 
\[
\frac{1}{\lambda^t}M^t \vec v_0 = v_t + \frac{\lambda'}{\lambda} \vec w + \vec w^*_t,
\]
where $\vec w^*_t$ is a (not necessarily non-negative)
 error term that satisfies $\| \vec w^*_t\| \leq\epsilon$.
 
Therefore we obtain
\[
\lim_{t\to \infty}  \frac{1}{\lambda^t} M^t \vec  v_{0} = \lambda' (\vec v + \frac{1}{\lambda} \vec w)\,,
\]
which proves the claim for $\vec w_0 = \frac{1}{\lambda} \vec w$.

\bigskip

\noindent{\bf Case 2:}
Assume that $\lambda_u = \lambda$. 
\smallskip

Similar to the previous case we first consider the 
extended PF-eigenvector $\vec{v}\in\cal{B}$ corresponding to the block $B$.
Recall that $\vec u \in \cal C$ is the vector given by the equation 
\[
M\vec v = \lambda \vec v + \vec u.
\]
 
We compute:
\begin{align*}
\frac{1}{\lambda_u^t \cdot t^{d_u+1}} M^t \vec v
&= \frac{1}{\lambda_u^t \cdot t^{d_u+1}} \left(\lambda^t \vec v + \sum_{j=0}^{t-1} \lambda^{j} M^{t-j-1} \vec u_{}\right)
\\&= \frac{1}{t^{d_u+1}} \vec v +  \frac{1}{\lambda}\sum_{j=0}^{t-1}  \frac{({t-j-1})^{d_u}}{t^{d_u+1}} \frac{1}{\lambda^{t-j-1} \cdot ({t-j-1})^{d_u}} M^{t-j-1} \vec u_{}\tag{$\dagger\dagger$}\label{secondsum}
\end{align*}
The
first term in this sum tends to $0$ when $t$ goes to infinity. In order to understand the limit of the second term in the above sum (\ref{secondsum}) we recall from the inductive hypothesis in Proposition \ref{induction-step} that the vectors
\[
\frac{1}{\lambda^{s} \cdot {s}^{d_u}} M^{s} \vec u_{}
\]
converge for $s \to \infty$ to some 
vector $\vec u_\infty$ in $\cal C$.

Since we need it later, we observe here that 
it follows from Lemma \ref{initial-faces} that some iterate $M^t \vec u$ belongs to the dominant interior of $\cal C\,$. Thus the inductive hypothesis in Proposition \ref{induction-step} states that $\vec u_\infty \neq \vec 0$ is an eigenvector of $M$.

In both cases, we derive that for any $\epsilon > 0$ there exists a bound $s(\epsilon)\ge0$ such that for all $s\geq s(\epsilon)$ we have
\[
\left\|\frac{1}{\lambda^{s} \cdot {s}^{d_u}} M^{s} \vec u_{} - \vec u_\infty\right\| \leq \epsilon\, ,
\]
from which we deduce that
\[
\left\|\frac{1}{\lambda^{t-j-1} \cdot ({t-j-1})^{d_u}} M^{t-j-1} \vec u_{} - \vec u_\infty\right\| \leq \epsilon
\]
holds for any $t-j-1 \geq s(\epsilon)$ or, equivalently, $j \leq t-s(\epsilon)-1$.

Thus we can split the second term in the above sum (\ref{secondsum})
as follows:
\begin{align*}
&\frac{1}{\lambda}\sum_{j=0}^{t-1}  \frac{({t-j-1})^{d_u}}{t^{d_u+1}} \frac{1}{\lambda^{t-j-1} \cdot ({t-j-1})^{d_u}} M^{t-j-1} \vec u_{}
\\&=
\frac{1}{\lambda}\sum_{j=0}^{t-s(\epsilon)-1}  \frac{({t-j-1})^{d_u}}{t^{d_u+1}} \frac{1}{\lambda^{t-j-1} \cdot ({t-j-1})^{d_u}} M^{t-j-1} \vec u_{}
\\&+
\frac{1}{\lambda}\sum_{j=t-s(\epsilon)}^{t-1}  \frac{({t-j-1})^{d_u}}{t^{d_u+1}} \frac{1}{\lambda^{t-j-1} \cdot ({t-j-1})^{d_u}} M^{t-j-1} \vec u_{}.
\end{align*}
For fixed $\epsilon>0$ and hence fixed $s(\epsilon)$ the second term in the last sum converges to 0 as $t$ tends to $\infty$, since 
\[
 \frac{({t-j-1})^{d_u}}{t^{d_u+1}}\le\frac{t^{d_u} }{t^{d_u+1}}\le\frac{1}{t}\, .
\] 

In order to compute the first term in (\ref{secondsum}) we observe that
\begin{align*}
&\left\|\sum_{j=0}^{t-s(\epsilon)-1}  \frac{({t-j-1})^{d_u}}{t^{d_u+1}} \frac{1}{\lambda^{t-j-1} \cdot ({t-j-1})^{d_u}} M^{t-j-1} \vec u_{}-
\sum_{j=0}^{t-s(\epsilon)-1}  \frac{({t-j-1})^{d_u}}{t^{d_u+1}}  \vec u_{\infty}\right\|
\\&\leq
\sum_{j=0}^{t-s(\epsilon)-1}\left \| (\frac{({t-j-1})^{d_u}}{t^{d_u+1}} \frac{1}{\lambda^{t-j-1} \cdot ({t-j-1})^{d_u}} M^{t-j-1} \vec u_{}
- \frac{({t-j-1})^{d_u}}{t^{d_u+1}}  \vec u_{\infty})\right\|
\\&=\sum_{j=0}^{t-s(\epsilon)-1}  \frac{({t-j-1})^{d_u}}{t^{d_u+1}} \left\|\frac{1}{\lambda^{t-j-1} \cdot ({t-j-1})^{d_u}} M^{t-j-1} \vec u_{}
 - \vec u_{\infty}\right\|
\\&\leq\sum_{j=0}^{t-s(\epsilon)-1}  \frac{({t-j-1})^{d_u}}{t^{d_u+1}} \epsilon\le\epsilon.
\end{align*}
This shows that 
\begin{align*}
\lim_{t\to\infty}\frac{1}{\lambda_u^t \cdot t^{d_u+1}} M^t \vec v&=
\frac{1}{\lambda}\lim_{t\to \infty} \sum_{j=0}^{t-s(\epsilon)-1}  \frac{({t-j-1})^{d_u}}{t^{d_u+1}}  \vec u_{\infty}
\\
&=\left(\frac{1}{\lambda}\lim_{t\to \infty} \sum_{j=0}^{t-s(\epsilon)-1}  \frac{({t-j-1})^{d_u}}{t^{d_u+1}}\right)\vec u_{\infty}\\
&=\left(\frac{1}{\lambda}\lim_{t\to \infty} \frac{1}{t^{d_u+1}}\sum_{k=s(\epsilon)}^{t-1}  {k}^{d_u}\right)\vec u_{\infty}.
\end{align*}
We note here that  
\[
\frac{1}{t^{d_u+1}}\sum_{k=0}^{t-1}  {k}^{d_u}\le1
\]
for all $t\ge1$. On the other hand,  \[\frac{1}{t^{d_u+1}}\sum_{k=0}^{t-1}  {k}^{d_u}\ge
\frac{1}{t^{d_u+1}}\sum_{k=t/2}^{t-1}  {k}^{d_u} \geq \frac{1}{t^{d_u+1}} \sum_{k=t/2}^{t-1}  (\frac{t}{2})^{d_u} = (\frac{1}{2})^{d_u+1}>0\]
for sufficiently large $t$,
so that, 
using the above observation that $\vec u_\infty \neq \vec 0$, we conclude that
the limit vector $\lambda_0 \vec u_\infty$ with
\[
\lambda_0 := \left(\frac{1}{\lambda}\lim_{t\to \infty} \frac{1}{t^{d_u+1}}\sum_{k=s(\epsilon)}^{t-1}  {k}^{d_u}\right)\tag{1}\label{lambda}
\]
is an 
eigenvector of $M$ in $\cal C$.
This proves the claim for the extended PF-eigenvector $\vec{v}$. \medskip

We now consider an arbitrary vector $\vec v_0 \in \cal B$, as well as the vectors  $v_t \in \cal B$ and $u_t \in \cal C$ as defined iteratively as before.
We obtain:
\begin{align*}
\frac{1}{\lambda_u^t \cdot t^{d_u+1}} M^t \vec v_0
&= \frac{1}{\lambda_u^t \cdot t^{d_u+1}} \left(\lambda^t \vec v_t + \sum_{j=0}^{t-1} \lambda^{j} M^{t-j-1} \vec u_{j+1}\right)
\\&= \frac{1}{t^{d_u+1}} \vec v_t +  \frac{1}{\lambda}\sum_{j=0}^{t-1}  \frac{({t-j-1})^{d_u}}{t^{d_u+1}} \frac{1}{\lambda^{t-j-1} \cdot ({t-j-1})^{d_u}} M^{t-j-1} \vec u_{j+1}\tag{$\ddagger$} \label{thirdsum}
\end{align*} 

The first term in this sum tends to $0$ when $t$ goes to infinity. In order to understand the limit of the second term 
we observe that the primitivity of the diagonal matrix block 
of $M$ corresponding to $B_{i}$ 
implies that the $\vec v_t$ converge to $\lambda' \vec v$ for some scalar $\lambda' > 0$. We write 
(as in Case 1)
$\vec u_{t+1} = \lambda' \vec u + \vec w_{t+1}$ and note that 
for any $\epsilon > 0$ there exists an integer $t_0 = t_0(\epsilon)$ such that
$\|w_{t+1}\| \leq \epsilon$ for any $t \geq t_0$.
As in Case 1 we have 
\[
\frac{1}{\lambda_u^t \cdot t^{d_u}} \|M^t \vec w_t\| \leq K(\epsilon)
\]
for all $t \geq t_0$ where $K(\epsilon)$ is the constant given by Lemma \ref{universal-constant}. 

As before, let $s(\epsilon)$ be an integer which ensures for all $s\geq s(\epsilon)$
that
\[
\|\frac{1}{\lambda^{s} \cdot {s}^{d_u}} M^{s} \vec u_{} - \vec u_\infty\| \leq \epsilon\, ,
\]
from which we deduce that
\[
\|\frac{1}{\lambda^{t-j-1} \cdot ({t-j-1})^{d_u}} M^{t-j-1} \vec u_{} - \vec u_\infty\| \leq \epsilon
\]
holds for any $t-j-1 \geq s(\epsilon)$ or, equivalently, $j \leq t-s(\epsilon)-1$.

We now 
split the second term in the above sum (\ref{thirdsum}) as follows:
\begin{align*}
&\frac{1}{\lambda}\sum_{j=0}^{t-1}  \frac{({t-j-1})^{d_u}}{t^{d_u+1}} \frac{1}{\lambda^{t-j-1} \cdot ({t-j-1})^{d_u}} M^{t-j-1} \vec u_{j+1}
\\&=
\frac{1}{\lambda}\sum_{j=0}^{t_0-1}  \frac{({t-j-1})^{d_u}}{t^{d_u+1}} \frac{1}{\lambda^{t-j-1} \cdot ({t-j-1})^{d_u}} M^{t-j-1} \vec u_{j+1}
\\&+
\frac{1}{\lambda}\sum_{j=t_0}^{t-s(\epsilon)-1}  \frac{({t-j-1})^{d_u}}{t^{d_u+1}} \frac{1}{\lambda^{t-j-1} \cdot ({t-j-1})^{d_u}} M^{t-j-1} \vec u_{j+1}
\\&+
\frac{1}{\lambda}\sum_{j=t-s(\epsilon)}^{t-1}  \frac{({t-j-1})^{d_u}}{t^{d_u+1}} \frac{1}{\lambda^{t-j-1} \cdot ({t-j-1})^{d_u}} M^{t-j-1} \vec u_{j+1}
\end{align*}
For a fixed $\epsilon>0$ and hence a fixed $t_0 = t_0(\epsilon)$ and $s(\epsilon)$, the first and the third term in the last sum converge to 0, as $t$ tends to $\infty$, since 
\[
 \frac{({t-j-1})^{d_u}}{t^{d_u+1}}\le\frac{t^{d_u} }{t^{d_u+1}}\le\frac{1}{t}\to0.
\] 
and the terms 
\[
\frac{1}{\lambda^{t-j-1} \cdot ({t-j-1})^{d_u}} M^{t-j-1} \vec u_{j+1}
\]
are uniformly bounded as we observed in Case 1. 

We now analyze the second term, where $\lambda'$ defined above through $\lim \vec v_t = \lambda' \vec v \,$:

\begin{align*}
&\left\|\sum_{j=t_0}^{t-s(\epsilon)-1}  \frac{({t-j-1})^{d_u}}{t^{d_u+1}} \frac{1}{\lambda^{t-j-1} \cdot ({t-j-1})^{d_u}} M^{t-j-1} \vec u_{j+1}
- \sum_{j=t_0}^{t-s(\epsilon)-1}  \frac{({t-j-1})^{d_u}}{t^{d_u+1}}  \lambda' \vec u_{\infty}\right\|
\\&\leq
\sum_{j=t_0}^{t-s(\epsilon)-1}\left \| (\frac{({t-j-1})^{d_u}}{t^{d_u+1}} \frac{1}{\lambda^{t-j-1} \cdot ({t-j-1})^{d_u}} M^{t-j-1} \vec u_{j+1}
- \frac{({t-j-1})^{d_u}}{t^{d_u+1}} \frac{1}{\lambda^{t-j-1} \cdot ({t-j-1})^{d_u}} M^{t-j-1} \lambda' \vec u_{})\right\|
\\&+
\sum_{j=t_0}^{t-s(\epsilon)-1}\left \| (\frac{({t-j-1})^{d_u}}{t^{d_u+1}} \frac{1}{\lambda^{t-j-1} \cdot ({t-j-1})^{d_u}} M^{t-j-1} \lambda' \vec u_{}
- \frac{({t-j-1})^{d_u}}{t^{d_u+1}}  \lambda' \vec u_{\infty})\right\|
\\&\leq
\sum_{j=t_0}^{t-s(\epsilon)-1} \frac{({t-j-1})^{d_u}}{t^{d_u+1}} \left \|\frac{1}{\lambda^{t-j-1} \cdot ({t-j-1})^{d_u}}  M^{t-j-1} \vec w_{j+1}
)\right\|
\\&+
\sum_{j=t_0}^{t-s(\epsilon)-1}\left \| \frac{({t-j-1})^{d_u}}{t^{d_u+1}} \frac{1}{\lambda^{t-j-1} \cdot ({t-j-1})^{d_u}} M^{t-j-1} \lambda' \vec u_{}
- \frac{({t-j-1})^{d_u}}{t^{d_u+1}}  \lambda' \vec u_{\infty})\right\|
\end{align*}
\begin{align*}
\\&\leq
\sum_{j=0}^{t-s(\epsilon)-1} (\frac{({t-j-1})^{d_u}}{t^{d_u+1}} 
K(\epsilon)
\\&+
\sum_{j=0}^{t-s(\epsilon)-1}  \frac{({t-j-1})^{d_u}}{t^{d_u+1}} \lambda' \left\|\frac{1}{\lambda^{t-j-1} \cdot ({t-j-1})^{d_u}} M^{t-j-1} \vec u_{}
 - \vec u_{\infty}\right\|
\\&\leq\sum_{j=0}^{t-s(\epsilon)-1}  \frac{({t-j-1})^{d_u}}{t^{d_u+1}} (1 + \lambda') K(\epsilon)
\\&
\leq (1 + \lambda') K(\epsilon),
\end{align*}
which tends to $0$ as $\epsilon\to 0$. 

Together with the previous estimates this shows that:
\[
\lim_{t \to \infty} \frac{1}{\lambda_u^t \cdot t^{d_u+1}} M^t \vec v_0
= \lambda' \lambda_0 \vec u_\infty, \]
where $\lambda_0 > 0$ is given by the Formula (\ref{lambda}) above. 
This finishes the proof in case 2, for $\lambda(\vec v_0) = \lambda' \lambda_0$.
We note that in this case 2 (as in case 1) all vectors in the block $\cal B$ have the same limit vector up to scaling. 

\smallskip
\noindent
{\bf Case 3:}
Assume that $\lambda_u > \lambda$.  

\smallskip
Note that this also includes the case where the diagonal block of $M$ corresponding to $B$ is a $(1 \times 1)$-matrix with entry $0$ or $1$.
 
For an arbitrary vector $\vec v_0$ of the cone $\cal B$ consider the following computation, 
where the vectors $\vec v_t \in \cal B$ and $\vec u_t \in \cal C$ are defined as before, and the value of the bound $t_0$ will be specified later:
\begin{align*}
\frac{1}{\lambda_u^t \cdot t^{d_u}} M^t \vec v_0
 &= \frac{1}{\lambda_u^t \cdot t^{d_u}}\left(\lambda^t \vec v_t + \sum_{j=0}^{t-1} \lambda^{j} M^{t-j-1} \vec u_{j+1}\right)\\
 &= \left(\frac{\lambda}{\lambda_u}\right)^t \frac{1}{t^{d_u}} \vec v_t +  \frac{1}{\lambda_u}\sum_{j=0}^{t-1} \left(\frac{\lambda}{\lambda_u}\right)^j \left(\frac{{t-j-1}}{t}\right)^{d_u} \frac{1}{\lambda_u^{t-j-1} \cdot ({t-j-1})^{d_u}} M^{t-j-1} \vec {u}_{j+1}\\
&= \tag{$\ddagger\ddagger$}\label{fourthsum}
\left(\frac{\lambda}{\lambda_u}\right)^t \frac{1}{t^{d_u}} \vec v_t +   
\frac{1}{\lambda_u}\sum_{j=0}^{t_0-1}  \left(\frac{\lambda}{\lambda_u}\right)^j \left(\frac{{t-j-1}}{t}\right)^{d_u} \frac{1}{\lambda_u^{t-j-1} \cdot ({t-j-1})^{d_u}} M^{t-j-1} \vec u_{j+1}\\
&+
\frac{1}{\lambda_u}\sum_{j=t_0}^{t-1}  \left(\frac{\lambda}{\lambda_u}\right)^j \left(\frac{{t-j-1}}{t}\right)^{d_u} \frac{1}{\lambda_u^{t-j-1} \cdot ({t-j-1})^{d_u}} M^{t-j-1} \vec u_{j+1} \, ,
 \end{align*}

The first term in the last sum tends to $0$ when $t$ goes to infinity. 
In order to understand the limit of the third term 
we argue (as in case 1) that from the definition of the $\vec v_t$ and $\vec u_t$ it follows directly that 
the values 
$\| \vec v_t\|$ and
$\| \vec u_t\|$ are uniformly bounded over all $t \geq 0$ by some constant $K_0 \geq 0$. 
Hence it follows from Lemma \ref{universal-constant} that there exist $K=K(K_0)$ such that 
\[
\|\frac{1}{\lambda_u^{t-j-1} \cdot ({t-j-1})^{d_u}} M^{t-j-1} \vec u_{j+1}\| < K
\]
for all $t \geq j \geq t_0$.
  
This immediately gives:
\begin{align*}
&\left\|\sum_{j=t_0}^{t-1} \left(\frac{\lambda}{\lambda_u}\right)^j \left(\frac{{t-j-1}}{t}\right)^{d_u} \frac{1}{\lambda_u^{t-j-1} \cdot ({t-j-1})^{d_u}} M^{t-j-1} \vec u_{j+1} 
\right\|\\
&\leq \sum_{j=t_0}^{t-1} \left(\frac{\lambda}{\lambda_u}\right)^j \left(\frac{{t-j-1}}{t}\right)^{d_u} \left \|\frac{1}{\lambda_u^{t-j-1} \cdot ({t-j-1})^{d_u}} M^{t-j-1} \vec u_{j+1}
\right\|\\
&\leq \sum_{j=t_0}^{t-1} \left(\frac{\lambda}{\lambda_u}\right)^j \left(\frac{{t-j-1}}{t}\right)^{d_u} K\\
&\leq K \sum_{j=t_0}^{t-1} \left(\frac{\lambda}{\lambda_u}\right)^j \left(\frac{{t-j-1}}{t}\right)^{d_u}  
\leq K \sum_{j=t_0}^{t-1} \left(\frac{\lambda}{\lambda_u}\right)^j   \\
&
 = K\left(\frac{\lambda}{\lambda_u}\right)^{t_0} \, \sum_{j = 0}^{t - t_0-1} \left(\frac{\lambda}{\lambda_u}\right)^j   
\leq K\,\frac{\lambda_u}{\lambda_u - \lambda}\left(\frac{\lambda}{\lambda_u}\right)^{t_0}
\end{align*}
This shows that the third term of the above sum converges for increasing $t_0$ to $0$.

\smallskip

In order to understand the limit of the second term of the above sum (\ref{fourthsum}) we recall from the inductive hypotheses that for any of the $\vec u_j$
\[
\lim_{s\to\infty}\frac{1}{\lambda_u^{s} \cdot {s}^{d_u}} M^{s} \vec{u}_j =\vec u_\infty^j
\] 
where either $\vec u_{\infty}^j = \vec 0$ or 
$\vec u_{\infty}^j$ of is an eigenvector of $M$ with eigenvalue $\lambda_u$. 
From Lemma \ref{initial-faces} and our induction hypothesis on vectors in the dominant interior of $\cal C$
we know that except for
a bounded number of small values of $j$ one has $\vec u_{\infty}^j \neq \vec 0$, 
and $\vec u_{\infty}^j$ is an eigenvector of $M$ with eigenvalue $\lambda_u$ (see Remark \ref{eigenvalue}).

Moreover, as we observed above,the values 
$\| \vec u_j\|$ are uniformly bounded over all $j \geq 0$. Hence for any $\epsilon > 0$ there exists a bound $s(\epsilon, j)$ such that for all $s\geq s(\epsilon, j)$ we have:
\[
\left\|\frac{1}{\lambda_u^{s} \cdot {s}^{d_u}} M^{s} \vec u_{j} -\vec u_\infty^j\right\| \leq \epsilon
\]
For any choice of $t_0 \geq 0$ we define 
$$
s_m(\epsilon, t_0) = \max\{ s(\epsilon, j)  \mid 1 \leq j \leq t_0 \}
$$
and thus obtain
\[
\left\|\frac{1}{\lambda_u^{t-j-1} \cdot ({t-j-1})^{d_u}} M^{t-j-1} \vec u_{j+1} - \vec u_\infty^{j+1}\right\| \leq \epsilon
\]
for any $0 \leq j \leq t_0 - 1$ and $t-j-1 \geq s_m$. This gives, for any $t \geq t_0 + s_m +1$:

\begin{align*}
&\left\|\sum_{j=0}^{t_0-1} \left(\frac{\lambda}{\lambda_u}\right)^j \left(\frac{{t-j-1}}{t}\right)^{d_u} \frac{1}{\lambda_u^{t-j-1} \cdot ({t-j-1})^{d_u}} M^{t-j-1} \vec u_{j+1} 
- \sum_{j=0}^{t_0-1}  \left(\frac{\lambda}{\lambda_u}\right)^j \left(\frac{{t-j-1}}{t}\right)^{d_u} \vec u_{\infty}^{j+1}\right\|\\
&\leq \sum_{j=0}^{t_0-1}\left \| \left(\frac{\lambda}{\lambda_u}\right)^j \left(\frac{{t-j-1}}{t}\right)^{d_u} \frac{1}{\lambda_u^{t-j-1} \cdot ({t-j-1})^{d_u}} M^{t-j-1} \vec u_{j+1}-
\left(\frac{\lambda}{\lambda_u}\right)^j \left(\frac{{t-j-1}}{t}\right)^{d_u} \vec u_{\infty}^{j+1})\right\|\\
&\leq \sum_{j=0}^{t_0-1}  \left(\frac{\lambda}{\lambda_u}\right)^j \left(\frac{{t-j-1}}{t}\right)^{d_u} \left\|\frac{1}{\lambda_u^{t-j-1} \cdot ({t-j-1})^{d_u}} M^{t-j-1} \vec u_{j+1}- 
  \vec u_{\infty}^{j+1}\right\|\\
&\leq\sum_{j=0}^{t_0-1} \left(\frac{\lambda}{\lambda_u}\right)^j \left(\frac{{t-j-1}}{t}\right)^{d_u} \epsilon\,\,\,\le\,\,\,\sum_{j=0}^{t_0-1} \left(\frac{\lambda}{\lambda_u}\right)^j\epsilon \\
&\leq \frac{\lambda_u}{\lambda_u - \lambda}
\epsilon.
\end{align*}
This shows that the term
\[
\frac{1}{\lambda_u} \sum_{j=0}^{t_0-1} \left(\frac{\lambda}{\lambda_u}\right)^j \left(\frac{{t-j-1}}{t}\right)^{d_u} \frac{1}{\lambda_u^{t-j-1} \cdot ({t-j-1})^{d_u}} M^{t-j-1} \vec u_{j+1}
\]
is $\frac{\lambda_u}{\lambda_u - \lambda}\epsilon$-close to the 
sum
\[
\frac{1}{\lambda_u} \sum_{j=0}^{t_0-1}  \left(\frac{\lambda}{\lambda_u}\right)^j \left(\frac{{t-j-1}}{t}\right)^{d_u}  \vec u_{\infty}^{j+1} \, ,
\]
which is non-zero for all sufficiently large $t_0$. 
Since it is the sum of eigenvectors with same eigenvalue $\lambda_u$, it is itself an eigenvector with eigenvalue $\lambda_u$.

\medskip

We now put together the arguments for the first, the second term and the third term of the above sum and obtain:
For $t \to \infty$ the family of vectors $\frac{1}{\lambda_u^t \cdot t^{d_u}} M^t \vec v_0$ converges
to the eigenvector
$$
\frac{1}{\lambda_u} \sum_{j=0}^{\infty}  \left(\frac{\lambda}{\lambda_u}\right)^j 
\vec u_{\infty}^{j+1}.
$$
The reader should notice that, contrary to the other two cases, in this case 3 this limiting eigenvector does depend on the choice of the ``starting vector'' $\vec v_0$.

\section{Eigenvectors and PB-Frobenius convergence}
\label{PB-convergence}

\subsection{Eigenvectors for matrices in primitive Frobenius form}
\label{primitive-F-eigenvectors}

Let $M$ be a non-negative integer square matrix in primitive Frobenius form
with no zero-columns. 
We say that a block $B_i$ in the associated block decomposition is {\em principal} if for every block $B_j$ in the dependency block union $C(B_i)$ the corresponding PF-eigenvalues satisfy:
$$\lambda_i > \lambda_j$$
This is equivalent to stating that any maximal chain of blocks $B_j$ that realize the growth type function of $\cal B_i + \cal C(B_i)$ (see the paragraph before Lemma \ref{initial-faces}) consists only of the single block $B_i$, i.e. the growth type function $h_i$ of $B_i$ is given by $h_i(t) = \lambda_i^t$.

\begin{lem}
\label{principal-ev}
Every principal block $B_i$ of $M$ determines an eigenvector $\vec v(B_i) \in \cal B_i + \cal C(B_i)$ with eigenvalue $\lambda_i$ which satisfies:
\begin{enumerate}
\item
The vector $\vec v(B_i)$ admits a decomposition 
$$\vec v(B_i) = \vec v^{\, \rm PF}_i + \vec w_i \, ,$$
where $\vec v^{\, \rm PF}_i$ is the 
extended PF-eigenvector (see Definition \ref{primitive-F} (3)) of the primitive diagonal block of $M$ corresponding to $B_i$, and $w_i \in \cal C(B_i)$.
\item
The vector $\vec v(B_i)$ is the only eigenvector in $\cal B_i + \cal C(B_i)$ which admits such a decomposition: Any other eigenvector in $\cal B_i + \cal C(B_i)$ is either contained in $\cal C(B_i)$, or else it is a scalar multiple of $\vec v(B_i)$. Hence, $\vec v(B_i)$ will be called the ``principal eigenvector'' of $B_i$ (or of $\cal B_i + \cal C(B_i)$).

\end{enumerate}
\end{lem}

\begin{proof}
Any non-zero vector $\vec v \in \cal B_i + \cal C(B_i)$ can be written as $\vec v = \vec v_0 + \vec u$, with $\vec v_0 \in \cal B$ and $\vec u \in \cal C(B_i)$. From the hypothesis that $B_i$ is principal it follows that the growth type of $\cal C(B_i)$ and thus that of $\vec u$ is strictly smaller than that of $B_i$, which is given by the function $h(t) = \lambda_i^t$. Case (1) of Proposition \ref{induction-step} thus shows that, if $\vec v_0 \neq \vec 0$, then $\frac{1}{h(t)} M^t(v_0)$ converges to a scalar multiple of 
the eigenvector 
$\vec v^{\, \rm PF}_i + \vec w_i$, where $w_i \in \cal C(B_i)$ is uniquely determined by the extended eigenvector $\vec v^{\, \rm PF}_i$. It follows directly that either 
$\vec v_0 = \vec 0$ and thus
$\vec v \in \cal C(B_i)$, or else
\[
\lim_{t\to\infty} \frac{1}{h(t)} M^t(\vec v) = \lambda(\vec v^{\, \rm PF}_i + \vec w_i)
\]
for some $\lambda > 0$.
In particular, we observe that any eigenvector in $\cal B_i + \cal C(B_i)$ which is not contained in $\cal C(B_i)$ must (up to rescaling) agree with $\vec v^{\, \rm PF}_i + \vec w_i$.
The latter is indeed an eigenvector with eigenvalue $\lambda_i$, by
Remark \ref{eigenvalue} and Lemma \ref{weak-interior} (1).

\end{proof}

We will denote by 
$\cal C(\lambda) \subset \cal C^n$ 
the non-negative cone 
spanned by all principal eigenvectors of $M$ with eigenvalue $\lambda$. As before, we write here $\cal C^n$ to denote the standard non-negative cone in $\R^n$. 
We also recall that for matrices in primitive Frobenius form there is a natural partial order on the blocks (see subsection \ref{block-matrices}), to which we refer below when a block is called ``minimal'' or ``maximal''.

\begin{prop}
\label{eigenvectors}
A 
vector $\vec v \in \cal C^n$ 
is an eigenvector of $M$ with eigenvalue $\lambda \geq 1$ if and only if $\vec v$ is contained in 
$\cal C(\lambda) \smallsetminus \{\vec 0\}$.
\end{prop}

\begin{proof}
Clearly any $\vec v \in \cal C(\lambda) \smallsetminus \{\vec 0\}$ is an eigenvector with eigenvalue $\lambda$. For the converse implication
we consider a maximal block $B$ of $M$, and assume by induction over the number of blocks in $M$ that the claim is true for the restriction of $M$ to the invariant block $\cal C$ spanned by all coordinate vectors not contained in $B$. If $B$ is not principal, it follows directly from the cases (2) and (3) of Proposition \ref{induction-step} that any eigenvector of $M$ must have zero entries in the coordinates that belong to $B$, so that the claim follows from the induction hypothesis.

Similarly, if $B$ is principal but the eigenvalue $\lambda$ of $\vec v$ is different from the PF-eigenvalue $\lambda_0$ of $B$, it follows from case (1) of Proposition \ref{induction-step} that $\vec v$ belongs to $\cal C$, so that the claim follows again from the induction hypothesis.

Finally, 
if $B$ is principal with $PF$-eigenvalue equal to $\lambda$, and with principal eigenvector $\vec v^{\, \rm PF} + \vec w$, then by the $M$-invariance of $\cal C$ we can apply Lemma \ref{principal-ev} to obtain a decomposition 
$$\vec v = \lambda' (\vec v^{\, \rm PF} + \vec w) + \vec u$$
for some vector $\vec u \in \cal C$ and some scalar $\lambda' \geq 0$. Since both, $\vec v$ and $\vec v^{\, \rm PF} + \vec w$ are eigenvectors with eigenvalue $\lambda$, the same is true for $\vec u$. Hence the claim follows again from our induction hypothesis.
${}^{}$
\end{proof}

\begin{rem}
\label{Schneider}
(1)
Eigenvectors of non-negative matrices have been investigated previously by several authors, see for instance \cite{ESS} and \cite{Ro} and the references given there.  Indeed, the 
statements 
of Lemma \ref{principal-ev} and Proposition \ref{eigenvectors} 
are very close to results obtained there.

In particular, in a slightly more general context, H. Schneider \cite{Sc} and his coauthors use, 
in the graph that canonically realizes the partial order on the set of irreducible matrix blocks of any non-negative matrix, 
the term ``distinguished'' for vertices
which correspond in the above considered case to what we call ``principle'' matrix blocks. They call the corresponding eigenvalues ``distinguished'', and any non-negative eigenvector is ``distinguished'' if it has a distinguished eigenvalue. Our ``principal'' eigenvectors would be, in their terminology, ``extremal distinguished'' eigenvectors, which are furthermore normalized.

\smallskip
\noindent
(2)
The reader should be aware of the fact that authors in dynamical systems use the attribute ``distinguished'' for eigenvectors in a slightly different meaning than what is common in applied linear algebra: In \cite{BKMS} as well as in \cite{HY}
``distinguished eigenvectors'' refers to what would be ``extremal distinguished eigenvectors'' in  Schneider's sense  above.

\smallskip
\noindent
(3)
As a final comment, we'd like to point out here that, similar to the above proofs of Lemma \ref{principal-ev} and Proposition \ref{eigenvectors}, an additional number of classical results (for instance Theorem 3.1 of \cite{Ro} or Theorem 3.7 of \cite{Sc}) about eigenvectors of non-negative matrices seem to follow as direct corollaries from our Proposition \ref{induction-step}.

\end{rem}

\medskip

\subsection{Eigenvectors for PB-Frobenius matrices}
\label{PB-eigenvectors}

We now turn our attention once again to non-negative matrices in PB-Frobenius form (see Definition \ref{Frobenius-form*}), as has been used throughout the first 3 sections of this paper. 

\begin{prop}
\label{PB-eigenvectors}
Let $M_0$ be a non-negative integer square matrix which is in PB-Frobenius form, and assume that 
$M_0$ is expanding (see Definition-Remark \ref{expanding-matrix}). 
Let $M_1$ be a positive power of $M_0$ which is
in primitive Frobenius form (with respect to a possibly refined block decomposition). Then every eigenvector of $M_1$ is also an eigenvector of $M_0$.
\end{prop}

\begin{proof}
We first note that the assumption that $M_0$ is expanding 
implies that $M_1$ has no zero-columns.

Since any two distinct principal eigenvectors of $M_1$ have non-zero coordinates in distinct principal blocks, it follows that they are linearly independent. Thus each principal eigenvector is an extremal point of the non-negative cone $\cal C(\lambda)$ spanned by all principal eigenvectors with same eigenvalue $\lambda$. Since the positive power $M_1$ of $M_0$ fixes every vector of $\cal C(\lambda)$ up to rescaling, it follows that $M_0$ must permute the principal eigenvectors of $M_1$ (up to rescaling).

We now observe that 
from the assumption that 
$M_0$ is expanding 
it follows 
furthermore 
that any minimal block of $M_1$ is primitive with PF-eigenvalue $> 1$. Thus from the definition of ``principal'' it follows that  all principal eigenvectors of $M_1$ have eigenvalue $> 1$.  Correspondingly, their associated principal block has as corresponding square diagonal matrix block a primitive matrix with PF-eigenvalue $> 1$. 
Recalling 
(see Remark \ref{iterating-to-get-primitive-F} (b)) 
that the block decomposition for $M_1$ is a refinement of the block decomposition for $M_0$, we deduce 
%
that these 
principal 
blocks can not be 
contained in a PB-block 
for $M_0$, so that by definition of the PB-Frobenius form they must be primitive blocks even for $M_0$.  In particular, each of them is fixed by $M_0$, which implies that the above permutation of $M_0$ of the principal eigenvectors of $M_1$ is trivial. Hence every primitive eigenvector of $M_1$ is also eigenvector of $M_0$, which implies the same for all of $\cal C(\lambda)$, thus proving our claim.
\end{proof}

We are now ready to prove the matrix convergence result stated in the Introduction: 

\begin{proof}[Proof of Theorem \ref{thmI}]
Let $M$ be the given matrix in PB-Frobenius form, 
which is assumed to be expanding. 
By Lemma \ref{Frobenius-powers*} there exists a positive power $M_1$ of $M$ which is in primitive Frobenius form. Let $\vec v \in \cal C$ be any non-zero vector, and apply Theorem \ref{GPFT} to get a limit eigenvector 
$$\vec v_\infty = \lim_{t\to\infty}\frac{M_1^{t}\vec{v}}{h_{\vec v}(t)}\, ,$$
for some normalization function $h_{\vec v}$ for the vector $\vec v$. The same statement (up to replacing $\vec v_\infty$ by a scalar multiple) stays valid if we replace $h_{\vec v}$ by any other normalization function for $\vec v$.
Thus in particular for the normalization function
(see Remark \ref{length-normalization})
$$h'_{\vec v}(t) = ||M^t \vec v ||$$
we want to consider the accumulation points of the values
$$\frac{M^{t}\vec{v}}{h'_{\vec v}(t)} \, .$$
As is true for all 
sequences of type $f^n(x)$ for which for some fixed $k$ the subsequence $f^{kn}(x)$ converges,
the sequence of 
vectors 
$\frac{M^{t}\vec{v}}{h'_{\vec v}(t)}$ 
must accumulate 
(up to rescaling) 
onto the 
finite 
$M$-orbit of $\underset{t\to\infty}{\lim}\frac{M_1^{t}\vec{v}}{h_{\vec v}(k t)}$, 
for $M_1 = M^k$. But from Proposition \ref{PB-eigenvectors} we know that this orbit consists 
(up to rescaling) 
only of a single point. 
Since by definition of $h'_{\vec v}$ we have $\|\frac{M^{t}\vec{v}}{h'_{\vec v}(t)}\| = 1$ for all $t \geq 1$, 
the family of vectors $\frac{M^{t}\vec{v}}{h'_{\vec v}(t)}$ 
must indeed converge.
\end{proof}

\bibliographystyle{alpha}
\bibliography{dynamicsoncurrents}

\end{document}